\documentclass[a4paper,11pt,twoside]{amsart}

\usepackage{graphicx}
\usepackage{amsfonts}
\usepackage{amsmath}
\usepackage{amsthm}
\usepackage{amssymb}
\usepackage[all]{xy}
\usepackage{hyperref}
\usepackage{aliascnt}

\newtheorem{thm}{Theorem}[section]

\newaliascnt{prop}{thm}
\newtheorem{prop}[prop]{Proposition}
\aliascntresetthe{prop}

\newaliascnt{lem}{thm}
\newtheorem{lem}[lem]{Lemma}
\aliascntresetthe{lem}

\newaliascnt{cor}{thm}
\newtheorem{cor}[cor]{Corollary}
\aliascntresetthe{cor}

\theoremstyle{definition}

\newaliascnt{definition}{thm}
\newtheorem{definition}[definition]{Definition}
\aliascntresetthe{definition}

\newaliascnt{remark}{thm}
\newtheorem{remark}[remark]{Remark}
\aliascntresetthe{remark}

\newaliascnt{ex}{thm}
\newtheorem{ex}[ex]{Example}
\aliascntresetthe{ex}

\newaliascnt{qn}{thm}
\newtheorem{qn}[qn]{Question}
\aliascntresetthe{qn}

\newaliascnt{pb}{thm}

\aliascntresetthe{pb}

\newaliascnt{conj}{thm}
\newtheorem{conj}[conj]{Conjecture}
\aliascntresetthe{conj}

\numberwithin{equation}{section}

\DeclareMathOperator{\im}{im} 
\DeclareMathOperator{\spec}{Spec} 
\newcommand{\iso}{\cong}

\newcommand{\farg}{-} 
\newcommand{\id}{\mathrm{id}}
\newcommand{\comp}{\circ} 
\newcommand{\mor}[1]{\xrightarrow{#1}}
\newcommand{\mono}{\hookrightarrow} 
\newcommand{\epi}{\twoheadrightarrow} 
\newcommand{\isomor}{\mor{\sim}} 
\newcommand{\card}[1]{\lvert#1\rvert} 
\newcommand{\rest}[1]{|_{#1}} 

\newcommand{\K}{\Bbbk} 
\newcommand{\cat}[1]{{\mathbf{#1}}} 
\newcommand{\opp}{^{\circ}} 
\newcommand{\s}[1]{\mathcal{#1}} 
\newcommand{\so}{\s{O}} 

\newcommand{\Ext}{\mathrm{Ext}}

\newcommand{\Iso}{\mathrm{Isom}}

\newcommand{\Br}{\mathrm{Br}}
\newcommand{\dg}{\mathrm{dg}}

\newcommand{\sh}[2][1]{#2[#1]} 
\newcommand{\FM}[2][]{\Phi^{#1}_{#2}} 
\newcommand{\D}[1][]{\mathrm{D}^{#1}} 
\newcommand{\Db}{\D[b]} 
\newcommand{\Dm}{\D[\text{-}]} 
\newcommand{\Dp}[1][]{\cat{Perf}_{#1}} 
\newcommand{\Dq}{\D_{\mathrm{qc}}} 
\newcommand{\Ddg}{\D_{\dg}} 
\newcommand{\Ddgb}{\D[b]_{\dg}} 
\newcommand{\dgD}{\mathcal{D}} 
\newcommand{\Perf}{\mathrm{Perf}^{\,\dg}} 
\newcommand{\rd}{\mathbf{R}} 
\newcommand{\ld}{\mathbf{L}} 
\newcommand{\lotimes}{\overset{\ld}{\otimes}} 
\newcommand{\fun}[1]{\mathsf{#1}} 
\newcommand{\Mod}[1]{\mathrm{Mod}(#1)} 
\newcommand{\dgMod}[1]{\mathrm{dgMod}(#1)}
\newcommand{\SF}[1]{\mathrm{SF}(#1)} 
\newcommand{\Coh}{\cat{Coh}}
\newcommand{\Qcoh}{\cat{Qcoh}}
\newcommand{\lto}{\longrightarrow}

\newcommand{\colim}[1]{\underset{#1}{\mathrm{colim}}\,}

\newcommand{\ZZ}{\mathbb{Z}}

\newcommand{\CC}{\mathbb{C}}

\newcommand{\FF}{\mathbb{F}}

\newcommand{\PP}{\mathbb{P}}

\newcommand{\Ind}{\fun{Ind}}
\newcommand{\Res}{\fun{Res}}
\newcommand{\Hqe}{\cat{Hqe}} 

\newcommand{\Plus}{\coprod}

\newcommand{\cA}{\cat{A}}
\newcommand{\cB}{\cat{B}}
\newcommand{\cC}{\cat{C}}
\newcommand{\cD}{\cat{D}}
\newcommand{\cF}{\cat{F}}
\newcommand{\cG}{\cat{G}}
\newcommand{\cL}{\cat{L}}

\newcommand{\cN}{\cat{N}}
\newcommand{\cP}{\cat{P}}
\newcommand{\cR}{\cat{R}}
\newcommand{\cS}{\cat{S}}
\newcommand{\cT}{\cat{T}}
\newcommand{\fE}{\fun{E}}
\newcommand{\fF}{\fun{F}}
\newcommand{\fG}{\fun{G}}
\newcommand{\fH}{\fun{H}}
\newcommand{\fI}{\fun{I}}

\newcommand{\fL}{\fun{L}}
\newcommand{\fM}{\fun{M}}
\newcommand{\fPo}{\fun{Y}'}

\newcommand{\ko}{\mathcal{O}}

\newcommand{\fQ}{\fun{Q}}
\newcommand{\Ho}{\mathrm{H}^0} 
\newcommand{\Zo}{\mathrm{Z}^0} 
\newcommand{\Yon}[1][\cA]{\fun{Y}^{#1}} 
\newcommand{\dgYon}[1][\cA]{\fun{Y}^{#1}_{\dg}} 

\newcommand{\Ac}{\mathrm{Ac}} 
\newcommand{\Acdg}{\mathrm{Ac}_{\dg}} 
\newcommand{\Acdgb}{\mathrm{Ac}_{\dg}^b} 

\newcommand{\St}[1]{\cat{S}(#1)} 
\newcommand{\ho}[1]{\mathrm{Ho}(#1)} 
\newcommand{\ep}{\varepsilon}
\newcommand{\cTr}{\cT_p}
\newcommand{\cTrc}{\cTr'}
\newcommand{\Cc}{\mathrm{C}} 
\newcommand{\Cdg}{\Cc_{\dg}} 
\newcommand{\Cdgb}{\Cc_{\dg}^b} 
\newcommand{\Cdgai}{\Cdg^{\mathrm{ai}}} 
\newcommand{\Kc}{\mathrm{K}} 
\newcommand{\Kai}{\Kc^{\mathrm{ai}}} 
\newcommand{\I}{\mathcal{I}} 
\newcommand{\X}{\mathcal{X}} 
\newcommand{\Y}{\mathcal{Y}} 
\newcommand{\diag}{\Delta} 
\newcommand{\adj}[4]{#1\colon#3\leftrightarrows#4\colon#2} 

\newcommand{\Spe}{\cat{Spectra}} 
\newcommand{\ExFun}{\cat{ExFun}} 
\newcommand{\rqr}[1]{\mathrm{h\text{-}proj}(#1)^{\mathrm{rqr}}}
\newcommand{\dgFun}{\underline{Hom}} 
\newcommand{\IHom}{\rd\dgFun}

\newcommand{\essim}[1]{\overline{#1}} 
\newcommand{\hproj}[1]{\mathrm{h\text{-}proj}(#1)} 
\newcommand{\hproja}[1]{\mathrm{h\text{-}proj}_\alpha(#1)} 
\newcommand{\dfun}[1]{\Phi^{\dg}_{#1}} 
\newcommand{\HH}{\mathrm{H\!H}}
\newcommand{\dgCat}{\cat{dgCat}}

\newcommand{\Inj}{\mathrm{Inj}} 
\newcommand{\Injdg}{\cat{Inj}} 
\newcommand{\Sh}{\cat{Sh}} 
\newcommand{\Fl}{\mathrm{Fl}}
\newcommand{\GL}{\mathrm{GL}}

\begin{document}

	\title[A tour about existence and uniqueness of dg enhancements and lifts]{A tour about existence and uniqueness of dg enhancements and lifts}

	\author{Alberto Canonaco and Paolo Stellari}

	\address{A.C.: Dipartimento di Matematica ``F. Casorati'', Universit{\`a}
	degli Studi di Pavia, Via Ferrata 5, 27100 Pavia, Italy}
	\email{alberto.canonaco@unipv.it}

	\address{P.S.: Dipartimento di Matematica ``F.
	Enriques'', Universit{\`a} degli Studi di Milano, Via Cesare Saldini
	50, 20133 Milano, Italy}
	\email{paolo.stellari@unimi.it}
    \urladdr{\url{http://users.unimi.it/stellari}}
	
	\thanks{A.~C.~ was partially supported by the national research project
	  ``Spazi di Moduli e Teoria di Lie'' (PRIN 2012).
	P.~S.~ is partially supported by the grant FIRB 2012 ``Moduli Spaces and Their Applications''.}

	\keywords{Dg categories, dg enhancements, triangulated categories}

	\subjclass[2010]{14F05, 18E10, 18E30}

\begin{abstract}
This paper surveys the recent advances concerning the relations between triangulated (or derived) categories and their dg enhancements. We explain when some interesting triangulated categories arising in algebraic geometry have a unique dg enhancement. This is the case, for example, for the unbounded derived category of quasi-coherent sheaves on an algebraic stack or for its full triangulated subcategory of perfect complexes. Moreover we give an account of the recent results about the possibility to lift exact functors between the bounded derived categories of coherent sheaves on smooth schemes to dg (quasi-)functors.
\end{abstract}

\maketitle

\setcounter{tocdepth}{1}
\tableofcontents

\section*{Introduction}

There are certainly many fruitful and interesting ways to enhance triangulated categories to higher categorical structures: differential graded (dg) categories, $A_\infty$-categories or stable $(\infty,1)$-categories, just to mention some. To some extent and in some appropriate sense, they are equivalent. We will not explore this, but it is certainly clear that as soon as the passage from triangulated to higher categories is achieved, then everything becomes very natural.

Indeed, triangulated and derived categories have acquired a growing relevance in the context of algebraic geometry. Their applications to the study of the geometry of moduli spaces are well consolidated, by now. Nonetheless, some aspects of their general theory look very unnatural. A well-known example is the so called `non-functoriality of the cone'. This problem disappears when one looks, for example, at (pretriangulated) dg categories: all (closed degree zero) morphisms have functorial cones. 

In this survey, we focus on dg categories and, more specifically, on some of their geometric incarnations. The kind of questions we are interested in have a foundational flavor and can be roughly summarized as follows:

\medskip
\centerline{\it When can a triangulated category be enhanced by a dg category?}

\centerline{\it Is a dg enhancement unique?}
\medskip

\noindent Clearly, the same kind of questions can be raised for exact functors between triangulated categories:

\medskip
\centerline{\it When can an exact functor be lifted to a dg quasi-functor?}

\centerline{\it Is a lift unique?}
\medskip

These problems become particularly interesting when the triangulated categories and the exact functors in question are of geometric nature. For example, one could consider the unbounded derived category of quasi-coherent sheaves $\D(\Qcoh(X))$ or the bounded derived category of coherent sheaves $\Db(X)$ or the category of perfect complexes $\Dp(X)$ on a scheme $X$ and exact functors between them.

In this setting the questions above have been around for a long while in the form of (folklore) conjectures. But only recently a satisfactory and (almost) complete set of answers has been provided. This is the reason why we believe that this is a good moment to collect and explain all the known results in this field. Let us now be a bit more precise.
\begin{itemize}
\smallskip
\item The triangulated categories $\D(\Qcoh(X))$, $\Db(X)$ and $\Dp(X)$, for any scheme $X$, have dg enhancements (see \autoref{subsect:exenh} and \autoref{subsect:prelfun}).
\end{itemize}
\smallskip
To be honest this is well-known for a long while and it is rather clear that finding geometric triangulated categories without dg enhancements needs inspiration from algebraic topology rather than from algebraic geometry (see \autoref{subsect:nonexistence}).

That the second part of the first question should have a positive answer was conjectured in \cite{BLL} (even in a stronger form). The first striking results in this direction are contained in the beautiful paper \cite{LO} by Lunts and Orlov. Here we discuss the vast generalization of their results contained in \cite{CS6} but the intellectual debt to the seminal paper by Lunts and Orlov has to be made clear at the very beginning. The results we want to discuss are the following:
\begin{itemize}
\smallskip
\item The triangulated categories $\D(\Qcoh(X))$, $\Db(X)$ and $\Dp(X)$ have unique dg enhancement at various levels of generality (see \autoref{subsect:applicunbounded}, \autoref{rmk:coh} and \autoref{cor:geocomp} respectively).
\end{itemize}
\smallskip
Nevertheless, there are categories with non-unique dg enhancements (see \autoref{subsect:nonuniqueness}).

\medskip

Now consider the case of an exact functor $\fF\colon\Db(X_1)\to\Db(X_2)$, where $X_i$ is, for example, a smooth projective scheme over a field. Again, the existence of a dg lift of $\fF$ was predicted in \cite{BLL,Or}. But here the situation becomes more sad. Indeed, in this case, the dg lift of $\fF$ may not exist (see \autoref{subsect:funcounter}) and may not be unique (see \autoref{subsect:funtria}). The beauty of this side of the story is that the property of being dg liftable does not depend on the choices of the dg enhancements (thanks to the uniqueness results mentioned above) and coincides with the fact that the exact functor is of Fourier--Mukai type. The precise definitions are given in \autoref{subsect:funtria} but the reader with some familiarity with Hodge theory should think of Fourier--Mukai functors as categorifications of the usual notion of a Hodge theoretic correspondence. Fourier--Mukai functors are ubiquitous in algebraic geometry and essentially all exact functors that an algebraic geometer encounters in his life are of this type. Nonetheless the recent counterexamples to the belief that all exact functors $\fF$ as above should be of this form raise a warning.

These facts make it clear that the relation between derived categories and their dg enhancements is quite delicate and cannot be superficially analyzed.

Parts of the considerations about Fourier--Mukai functors discussed in this paper have been already explained in \cite{CS3}. In a sense, the present paper should be thought of as a companion and an update of \cite{CS3}.

\medskip

Although the main problems of geometric nature have been solved by now, we believe that there are still very interesting and open problems that deserve attention. They are spread out all along this survey in the form of open questions and we hope that they may help stimulate further work on the subject.

\subsection*{Geometric motivation}

One of the goals of this survey is to try to attract the attention of algebraic geometers and to convince them that the issues we discuss here are of some interest from their viewpoint. Thus we feel some pressure to discuss at least one geometric application.

The homological version of the so called \emph{Mirror Symmetry Conjecture} by Kontsevich \cite{Ko} predicts that there should be an $A_\infty$-equivalence between a dg enhancement of $\Db(X)$, for $X$ a smooth projective Calabi--Yau threefold, and the Fukaya category of the mirror $Y$ of $X$. The Fukaya category is actually an $A_\infty$-category but its homotopy category is an ordinary triangulated category. Roughly speaking, the objects of this $A_\infty$-category are Lagrangian submanifolds of $Y$.

Now, one could ask the following natural questions:
\begin{itemize}
\item[(1)] Do we have to care about a specific choice of a dg enhancement of $\Db(X)$?
\item[(2)] Is the $A_{\infty}$-nature of the quasi-equivalence relevant? Or does only the induced exact equivalence between the (triangulated) homotopy categories matter?
\end{itemize}
The results concerning the questions at the beginning of this introduction and which are discussed in this paper suggest that, a posteriori and in theory, only the triangulated side matters. On the other hand, in practice, the situation is not so simple: in all the examples, the construction of the exact equivalence exploits the dg and the $A_{\infty}$-structures.

To get another example of why one should care about the uniqueness of dg enhancements, let us go back to the problem of showing that a Fourier--Mukai functor can be lifted to a dg quasi-functor. These triangulated functors are compositions of three exact functors (pull-back, push-forward and tensor product). Each of them needs to be derived using appropriate (and a priori distinct) dg enhancements. Only quite recently, Schn\"urer constructed in \cite{Sn} dg resolution functors for suitable dg categories over a field that allow to lift the three exact functors mentioned above to the same type of dg enhancement. On the other hand, the uniqueness of dg enhancements makes the problem well-posed in complete generality.

\subsection*{Plan of the paper}

The paper is organized as follows. In \autoref{sect:dgenhancements} we summarize some basic material about dg categories, dg functors, localizations of the category of dg categories and enhancements. The emphasis is on the geometric examples.

\autoref{sect:wellgen} is a short summary about well generated triangulated categories. Such a theory plays a crucial role in the proof of uniqueness of dg enhancements for the unbounded derived category of a Grothendieck category (this includes the case of $\D(\Qcoh(X))$ for any scheme $X$). This will be extensively discussed in \autoref{sect:uniqcat} but before indulging in this analysis we show that there are triangulated categories without dg enhancements or with more than one up to quasi-equivalence (see \autoref{sect:badcat}).

The uniqueness of the dg enhancements for the full subcategory of compact objects in the derived category of a Grothendieck category is explained in \autoref{sect:uniqcatcomp}. From this, we deduce all the results concerning the triangulated categories $\Db(X)$ and $\Dp(X)$.

Finally, in \autoref{sect:functors}, we study the problem of lifting exact functors by linking it to the quest of a characterization of such functors in terms of Fourier--Mukai functors.

\subsection*{Notation}

We assume that a universe containing an infinite set is fixed. Several definitions concerning dg categories need special care because they may, in principle, require a change of universe. All possible subtle logical issues in this sense can be overcome in view of \cite[Appendix A]{LO}. The careful reader should have a look at it. After these warnings and to simplify the notation, we will not mention explicitly the universe we are working in any further in the paper, as it should be clear from the context. The members of this universe will be called small sets. For example, when we speak about small coproducts in a category, we mean coproducts indexed by a small set. If not stated otherwise, we always assume that each Hom-space in a category forms a small set. A category is called \emph{small} if the isomorphism classes of its objects form a small set.

Given a category $\cC$ and two objects $C_1$ and $C_2$ in $\cC$, we denote by $\cC(C_1,C_2)$ the Hom-space between $C_1$ and $C_2$. If $\fF\colon\cC\to\cD$ is a functor and $C_1$ and $C_2$ are objects of $\cC$, then we denote by $\fF_{C_1,C_2}$ the induced map $\cC(C_1,C_2)\to\cD(\fF(C_1),\fF(C_2))$.

Unless otherwise stated, all categories and functors are assumed to be $\K$-linear, for a fixed commutative ring $\K$. By a $\K$-linear category we mean a category whose Hom-spaces are $\K$-modules and such that the compositions are $\K$-bilinear, not assuming that finite coproducts exist. Moreover, schemes and algebraic stacks are assumed to be defined over $\K$.

If $\cA$ is a small ($\K$-linear) category, we denote by $\Mod{\cA}$ the (abelian) category of ($\K$-linear) functors $\cA\opp\to\Mod{\K}$, where $\cA\opp$ is the opposite category of $\cA$ and $\Mod{\K}$ is the category of $\K$-modules. The Yoneda embedding of $\cA$ is the functor $\Yon\colon\cA\to\Mod{\cA}$ defined on objects by $A\mapsto\cA(\farg,A)$.

If $\cT$ is a triangulated category and $\cS$ a full triangulated subcategory of $\cT$, we denote by $\cT/\cS$ the Verdier quotient of $\cT$ by $\cS$. In general, $\cT/\cS$ is not a category according to our convention (namely, the Hom-spaces in $\cT/\cS$ need not be small sets), but it is in many common situations, for instance when $\cT$ is small. 

For a complex of objects in an abelian category
\[
C=\{\cdots\to C^{j-1}\mor{d^{j-1}}C^j\mor{d^j}C^{j+1}\to\cdots\},
\]
and for every integer $n$, we define
\[
\tau_{\le n}(C):=\{\cdots\to C^j\mor{d^j}C^{j+1}\to\cdots\to C^{n-1}\to\ker d^n\to0\to\cdots\}.
\]

\section{Dg categories and dg enhancements}\label{sect:dgenhancements}

This section provides a quick introduction to dg categories and to their geometric incarnations. The expert reader can certainly skip this section.

\subsection{Generalities on dg categories}\label{subsect:gendg}

Very well-written surveys about dg categories are \cite{K} and \cite{TLec}. Notice that `dg' is an acronym for differential graded.

\subsubsection*{Dg categories, dg functors and quasi-functors}

Starting from scratch, let us recall the following.

\begin{definition}\label{def:dgcat}
A \emph{dg category} is a category $\cC$
such that, for all $A,B,C$ in $\cC$, the morphism spaces
$\cC\left(A,B\right)$ are $\ZZ$-graded $\K$-modules with a differential
$d\colon\cC\left(A,B\right)\to\cC\left(A,B\right)$ of degree $1$ and the composition maps $\cC(B,C)\otimes_{\K}\cC(A,B)\to\cC(A,C)$, $g\otimes f\mapsto g\comp f$,
are morphisms of complexes.
\end{definition}

More explicitly, if $f$ and $g$ as above are homogeneous, then $\deg(g\comp f)=\deg(g)+\deg(f)$ and $d(g\comp f)=d(g)\comp f+(-1)^{\deg(g)}g\comp d(f)$. Moreover,
it follows from the definition that the identity of each object is a closed morphism of degree $0$.

\begin{ex}\label{ex:dgcat1}
(i) Any ($\K$-linear) category has a trivial structure of dg category,
with morphism spaces concentrated in degree $0$.

(ii) A dg category with one object is just a dg algebra.

(iii) For a dg category $\cC$, one defines the opposite dg category $\cC\opp$ with the same objects as $\cC$ while $\cC\opp(A,B):=\cC(B,A)$. Observe that, for homogeneous elements, the composition $f\comp g$ in $\cC\opp$ differs from the composition $g\comp f$ in $\cC$ by a factor $(-1)^{\deg(f)\deg(g)}$.

(iv) If $\cC$ is a dg category and $\cB$ is a full subcategory of $\cC$ (regarded as an ordinary category), then $\cB$ is in a natural way a dg category. Hence we will say that $\cB$ is a \emph{full dg subcategory} of $\cC$.

(v) Following \cite{Dr}, given a (small) dg category $\cC$ and a full dg
subcategory $\cB$ of $\cC$, one can define the quotient
$\cC/\cB$ which is again a dg category. The construction can be roughly summarized as follows. Given $B\in\cB$, we formally add a morphism $f_B\colon B\to B$ of degree $-1$ and we define $d(f_B)=\id_B$. To make this rigorous one would need to substitute $\cC$ and $\cB$ with cofibrant replacements. Since this level of precision is not needed in this paper, we ignore this issue.
\end{ex}

Given a dg category $\cC$ we denote by $\Zo(\cC)$ its \emph{underlying category} and by $\Ho(\cC)$ its \emph{homotopy category}. To be precise, the objects of both categories are the same as those of $\cC$ while the morphisms from $A$ to $B$ are given, respectively, by $Z^0(\cC(A,B))$ and $H^0(\cC(A,B))$.

\begin{ex}\label{Cdg}
If $\cA$ is a ($\K$-linear) category, we denote by $\Cdg(\cA)$ the dg category of complexes in $\cA$. More precisely, its objects are complexes in $\cA$ and morphisms are given (for all $A,B$ in $\Cdg(\cA)$) by
\[
\Cdg(\cA)(A,B)^n:=\prod_{i\in\ZZ}\cA(A^i,B^{n+i})
\]
with differential defined on homogeneous elements by $d(f):=d_B\comp f-(-1)^{\deg(f)}f\comp d_A$ (the composition is the obvious one). It is then straightforward to check that $\Cc(\cA):=\Zo(\Cdg(\cA))$ is the usual category of complexes in $\cA$ and $\Kc(\cA):=\Ho(\Cdg(\cA))$ is the usual homotopy category of complexes in $\cA$.
\end{ex}

\begin{definition}
A \emph{dg functor} $\fF\colon\cC_1\to\cC_2$ between two dg categories is a functor such that the map $\fF_{A,B}$ is a morphism of complexes of $\K$-modules, for every $A,B\in\cC_1$.
\end{definition}
  
We denote by $\dgCat$ the category whose objects are small dg categories and whose morphisms are dg functors. Clearly, a dg functor $\fF\colon\cC_1\to\cC_2$ induces a functor $\Ho(\fF)\colon\Ho(\cC_1)\to\Ho(\cC_2)$.

Starting from two dg categories $\cC_1$ and $\cC_2$, one can construct two new dg categories $\cC_1\otimes\cC_2$ and $\dgFun(\cC_1,\cC_2)$ (see \cite[Section 2.3]{K}). The objects of $\cC_1\otimes\cC_2$ are pairs $(A,B)$ with $A\in\cC_1$ and $B\in\cC_2$. The morphisms are defined by
\[
\cC_1\otimes\cC_2((A_1,B_1),(A_2,B_2))=\cC_1(A_1,A_2)\otimes_{\K}\cC_2(B_1,B_2),
\]
for all $(A_i,B_i)\in\cC_1\otimes\cC_2$ and $i=1,2$. The objects of $\dgFun(\cC_1,\cC_2)$ are dg functors from $\cC_1$ to $\cC_2$ and morphisms are given by suitably defined (dg) natural transformations. Notice that here we are slightly abusing the standard terminology. Indeed, one usually calls dg natural transformations only the closed degree zero morphisms in $\dgFun(\cC_1,\cC_2)$.

Notice that the tensor product defines a symmetric monoidal structure on $\dgCat$. This means that, up to isomorphism, the tensor product is associative, commutative and $\K$ acts as the identity. Moreover, given three dg categories $\cC_1$, $\cC_2$ and $\cC_3$, there is a natural isomorphism in $\dgCat$
\begin{equation}\label{dgadj}
\dgFun(\cC_1\otimes\cC_3,\cC_2)\iso\dgFun(\cC_1,\dgFun(\cC_3,\cC_2)).  
\end{equation}
By considering the objects of the two dg categories above, one obtains a natural bijection
\[
\dgCat(\cC_1\otimes\cC_3,\cC_2)\iso\dgCat(\cC_1,\dgFun(\cC_3,\cC_2)),
\]
which proves that the symmetric monoidal structure on $\dgCat$ is closed.

\begin{definition}
A dg functor $\fF\colon\cC_1\to\cC_2$ is a \emph{quasi-equivalence}, if the maps
$\fF_{A,B}$ are quasi-isomorphisms, for every $A,B\in\cC_1$, and $\Ho(\fF)$ is an equivalence.
\end{definition}

One can consider the localization $\Hqe$ of $\dgCat$ with respect to quasi-equivalences. Given a dg functor
$\fun{F}$, we will denote by the same symbol its image in
$\Hqe$.  We will adopt a rather non-standard terminology and call \emph{quasi-functor} any morphism in $\Hqe$. Indeed, more appropriately and according to \cite{Dr,LO}, a quasi-functor would be a suitable bimodule and any morphism in $\Hqe$ would be an isomorphism class of quasi-functors. Any quasi-functor (in our sense) $\fF\colon\cC_1\to\cC_2$ induces a functor
$\Ho(\fF)\colon\Ho(\cC_1)\to\Ho(\cC_2)$, which is well defined up to isomorphism. In the rest of the paper, we will treat $\Ho(\fF)$ as an actual functor rather than as an isomorphism class of functors. We are allowed to do that because everything will be independent of the choice of the representative in the isomorphism class.

\begin{remark}\label{rmk:modelcat}
It is worth mentioning here that $\dgCat$ is in a natural way a model category (see \cite{Tab}) and $\Hqe$ is its homotopy category.

Recall that a \emph{model category} (see \cite{Ho}) is a (not necessarily $\K$-linear) category which has small limits and colimits and which is endowed with three families of morphisms, called \emph{fibrations}, \emph{cofibrations} and \emph{weak equivalences}, satisfying some axioms which we are not going to list here. The \emph{homotopy category} of a model category $\cC$ is the localization $\ho{\cC}$ of $\cC$ with respect to weak equivalences.
\end{remark}

The tensor product in $\dgCat$ can be appropriately derived (see \cite{K} or \cite{To}). We will denote by $\cC_1\lotimes\cC_2$ the derived tensor product in $\Hqe$ of the dg categories $\cC_1$ and $\cC_2$. This defines a symmetric monoidal structure on $\Hqe$.

\subsubsection*{Dg modules and their relatives}

Given a small dg category $\cC$, one can consider the dg category $\dgMod{\cC}$ of \emph{(right) $\cC$-dg modules}, which is defined as $\dgFun(\cC\opp,\Cdg(\Mod{\K}))$. A $\cC$-dg module is \emph{representable} if it is contained
in the image of the Yoneda dg functor
\[
\dgYon[\cC]\colon\cC\to\dgMod{\cC}\qquad
A\mapsto\cC\left(\farg,A\right).
\]
It is known that $\Ho(\dgMod{\cC})$ is, in a natural way, a triangulated category with small coproducts (see, for example, \cite{K}). Hence, we can give the following definition.

\begin{definition}\label{def:pretr}
A dg category $\cC$ is \emph{pretriangulated} if the essential image of the functor
\[
\Ho(\dgYon[\cC])\colon\Ho(\cC)\to\Ho(\dgMod{\cC})
\]
is a triangulated subcategory.
\end{definition}

\begin{remark}\label{pretr}
(i) As $\Ho(\dgYon[\cC])$ is fully faithful, it is clear that $\Ho(\cC)$ is a triangulated category if $\cC$ is a pretriangulated dg category. Moreover, given a quasi-functor $\fF\colon\cC_1\to\cC_2$ between two pretriangulated dg categories, the induced functor $\Ho(\fF)\colon\Ho(\cC_1)\to\Ho(\cC_2)$ is an exact functor between triangulated categories.

(ii) Let $\cC$ be a pretriangulated dg category and let $\cB$ be a full pretriangulated dg subcategory of $\cC$. Then the dg quotient $\cC/\cB$ is again pretriangulated and there is a natural exact equivalence between the Verdier quotient $\Ho(\cC)/\Ho(\cB)$ and $\Ho(\cC/\cB)$ (see \cite{Dr} or \cite[Lemma 1.5]{LO}).
\end{remark}

The full dg subcategory of $\dgMod{\cC}$ of \emph{acyclic dg modules} is denoted by $\Ac(\cC)$. Here an object $N\in\dgMod{\cC}$ is acyclic if $N(C)$ is an acyclic complex, for all $C\in\cC$. The category $\Ho(\Ac(\cC))$ is a
localizing subcategory of the homotopy category
$\Ho(\dgMod{\cC})$. Recall that, given a triangulated category $\cT$ with small coproducts, a full triangulated subcategory is \emph{localizing} if it is closed under small coproducts in $\cT$.

The \emph{derived category} of the dg
category $\cC$ is the Verdier quotient
\[
\dgD(\cC):=\Ho(\dgMod{\cC})/\Ho(\Ac(\cC)),
\]
which turns out to be a triangulated category with small coproducts. By \autoref{pretr} (ii), there is a natural exact equivalence
\begin{equation}\label{Drinfeld}
\dgD(\cC)=\Ho(\dgMod{\cC})/\Ho(\Ac(\cC))\iso\Ho(\dgMod{\cC}/\Ac(\cC)).
\end{equation}
A $\cC$-dg module $M\in\dgMod{\cC}$ is \emph{h-projective} if $H^0(\dgMod{\cC})(M,N)=0$, for all $N\in\Ac(\cC)$. We denote by $\hproj{\cC}$ the full dg subcategory of $\dgMod{\cC}$ with objects the h-projective $\cC$-dg modules. It is easy to see that $\hproj{\cC}$ contains the full dg subcategory $\essim{\cC}$ of $\dgMod{\cC}$ whose objects are those in the essential image of $\Ho(\dgYon[\cC])$. Summing up, given a dg category $\cC$, we have the following inclusions of full dg subcategories:
\[
\dgYon[\cC](\cC)\mono\essim{\cC}\mono\hproj{\cC}\mono\dgMod{\cC}.
\]

\begin{remark}\label{rmk:pretr}
(i) There is also a notion of \emph{semi-free} dg module, which we do not define here, since we are not going to use it in the paper. It can be useful to know, however, that the full dg subcategory $\SF{\cC}$ of $\dgMod{\cC}$ consisting of semi-free dg modules is contained in $\hproj{\cC}$ and the inclusion $\SF{\cC}\mono\hproj{\cC}$ is a quasi-equivalence (see, for example, \cite{Ke1}). Clearly this means that it is essentially equivalent to work with $\SF{\cC}$ or $\hproj{\cC}$. On the other hand, sometimes one can prefer to use the former, because for some computations it can be easier to deal with semi-free dg modules (this is the case in \cite{LO} and \cite{CS6}).

(ii) It is easy to see that the homotopy category $\Ho(\hproj{\cC})$ is a full triangulated subcategory of $\Ho(\dgMod{\cC})$. Moreover, by \cite{Ke1}, there is an exact equivalence of triangulated categories $\Ho(\hproj{\cC})\iso\dgD(\cC)$. We can actually be more precise about it. Indeed, the composition of natural dg functors
\[
\fH\colon\hproj{\cC}\mono\dgMod{\cC}\to\dgMod{\cC}/\Ac(\cC)
\]
is a quasi-equivalence. So, up to composing with \eqref{Drinfeld}, $\Ho(\fH)$ provides the exact equivalence $\Ho(\hproj{\cC})\iso\dgD(\cC)$ mentioned above.
\end{remark}

For two dg categories $\cC_1$ and $\cC_2$, it follows from \eqref{dgadj} that there is an isomorphism of dg categories
$\dgMod{\cC_1\opp\otimes\cC_2}\iso\dgFun(\cC_1,\dgMod{\cC_2})$, so in
particular an object $E\in\dgMod{\cC_1\opp\otimes\cC_2}$ corresponds to
a dg functor $\dfun{E}\colon\cC_1\to\dgMod{\cC_2}$. Conversely, for
every dg functor $\fF\colon\cC_1\to\dgMod{\cC_2}$ there exists a unique
$E\in\dgMod{\cC_1\opp\otimes\cC_2}$ such that $\dfun{E}=\fF$. An
object $E\in\hproj{\cC_1\opp\otimes\cC_2}$ is called \emph{right
quasi-representable} if $\dfun{E}(\cC_1)\subset\essim{\cC_2}$. The
full dg subcategory of $\hproj{\cC_1\opp\otimes\cC_2}$ consisting of all
right quasi-representable dg modules will be denoted by
$\rqr{\cC_1\opp\otimes\cC_2}$.

If we are given a dg functor  $\fF\colon\cC_1\to\cC_2$, there exist dg functors
\[
\Ind(\fF)\colon\dgMod{\cC_1}\to\dgMod{\cC_2},
\qquad\Res(\fF)\colon\dgMod{\cC_2}\to\dgMod{\cC_1}.
\]
While $\Res(\fF)$ is simply defined by $\fM\mapsto\fM\comp\fF\opp$, the reader can have a look at \cite[Sect.\ 14]{Dr} for the explicit
definition and properties of $\Ind(\fF)$. Let us just observe that $\Ind(\fF)$ preserves h-projective dg modules and
$\Ind(\fF)\colon\hproj{\cC_1}\to\hproj{\cC_2}$ is a quasi-equivalence if
$\fF\colon\cC_1\to\cC_2$ is such. Moreover, $\Ind(\fF)$ commutes with
the Yoneda embeddings, up to dg isomorphism.

\begin{ex}\label{ResSF}
Let $\cC$ be a dg category and let $\cB$ be a full dg subcategory of
$\cC$. If we denote by $\fI\colon\cB\mono\cC$ the inclusion dg functor,
then the composition of dg functors
\[
\cC\mor{\dgYon[\cC]}\dgMod{\cC}\mor{\Res(\fI)}\dgMod{\cB}\to\dgMod{\cB}/\Ac(\cB)
\]
yields, in view of \autoref{rmk:pretr} (ii), a natural quasi-functor $\cC\to\hproj{\cB}$.
\end{ex}

\subsubsection*{Dg enhancements}

Let us give now the key definition for this paper.

\begin{definition}\label{def:enhancement}
A \emph{dg enhancement} (or simply an \emph{enhancement}) of a triangulated category $\cT$ is a pair
$(\cC,\fE)$, where $\cC$ is a pretriangulated
dg category and $\fE\colon\Ho(\cC)\to\cT$ is an exact
equivalence.
\end{definition}

When the equivalence $\fE$ is not relevant, by abuse of notation, we will often simply say that $\cC$ is an enhancement of $\cT$ if there exists an exact equivalence $\Ho(\cC)\iso\cT$.

\begin{definition}
A triangulated category is \emph{algebraic} if it has an enhancement.
\end{definition}

\begin{remark}
As we will see in \autoref{algchar}, there are other equivalent definitions of algebraic triangulated categories, which are used more often in the literature.
\end{remark}

Let us discuss a few easy examples now and elaborate a bit more on the subject in \autoref{subsect:exenh}.

\begin{ex}\label{ex:derenh}
(i) By \autoref{rmk:pretr} (ii), if $\cC$ is a dg category, $\hproj{\cC}$
is an enhancement of $\dgD(\cC)$.

(ii) If $\cC$ is a pretriangulated dg category and $\cB$ is a full pretriangulated dg subcategory of $\cC$, then, by \autoref{pretr} (ii), $\cC/\cB$ is an enhancement of $\Ho(\cC)/\Ho(\cB)$.

(iii) If $\cA$ is an additive category, then $\Cdg(\cA)$ is pretriangulated and it is naturally an enhancement of the triangulated category $\Kc(\cA)$ (see \autoref{Cdg}).
\end{ex}

The key point of this paper is that, in principle, one may have `different' enhancements for the same triangulated category. Let us start describing the weakest form of uniqueness of dg enhancements.

\begin{definition}\label{def:uniqueenh}
An algebraic triangulated category $\cT$
\emph{has a unique enhancement} if, given two enhancements $(\cC,\fE)$ and $(\cC',\fE')$ of $\cT$,
there exists a quasi-functor $\fF\colon\cC\to\cC'$ such that
$\Ho(\fF)$ is an equivalence (which implies that $\fF$ is an isomorphism in $\Hqe$).
\end{definition}

In the definition, the equivalence $\fE$, which is a key ingredient in \autoref{def:enhancement}, is not relevant. Thus it makes sense to say that a triangulated category has a unique dg enhancement if all its enhancements are isomorphic in $\Hqe$. On the other hand, there are stronger versions of the notion of uniqueness of dg enhancements.

\begin{definition}\label{def:strongenhancement}
An algebraic triangulated category $\cT$ \emph{has a strongly unique} (respectively, \emph{semi-strongly unique}) \emph{enhancement} if moreover $\fF$ can be chosen so that there is an isomorphism of exact functors $\fE\iso\fE'\comp\Ho(\fF)$ (respectively, there is an isomorphism $\fE(C)\iso\fE'(\Ho(\fF)(C))$ in $\cT$, for every $C\in\cC$).
\end{definition}

\subsection{Enhancements in geometric contexts}\label{subsect:exenh}

For simplicity, we will assume that $X$ is a smooth projective scheme defined over a field $\K$. These assumptions will be weakened along the paper. We denote by $\Qcoh(X)$ and $\Coh(X)$ the abelian categories of quasi-coherent and coherent sheaves on $X$. We use the short hand notation $\Db(X)$ for $\Db(\Coh(X))$. 

\subsubsection*{Category of complexes}\label{subsubsect:compl}

Let us first describe a rather general procedure. Assume that $\cA$ is an abelian category with the additional requirement that $\D(\cA)$ is a category (notice that this is not always the case, as pointed out in \cite[Section 4.15]{K2}).

Now, one can consider the (pretriangulated) full dg subcategory $\Acdg(\cA)$ of $\Cdg(\cA)$ whose objects are the acyclic complexes. Then, in view of \autoref{ex:derenh} (ii) and (iii), the dg quotient category $\Ddg(\cA):=\Cdg(\cA)/\Acdg(\cA)$ is an enhancement of $\D(\cA)$. Clearly, one could also work with the dg categories $\Cdgb(\cA)$ and $\Acdgb(\cA)$ of bounded complexes, obtaining an enhancement $\Ddgb(\cA)$ of $\Db(\cA)$.

In conclusion, $\Ddg(\Qcoh(X))$ and $\Ddgb(\Coh(X))$, together with the corresponding natural exact equivalences with $\D(\Qcoh(X))$ and $\Db(X)$, are enhancements of the latter triangulated categories.

\subsubsection*{Injective resolutions}\label{subsubsect:inj}

On a smooth projective scheme, one can consider the dg category $\Injdg(X)$ of bounded below complexes of injective quasi-coherent sheaves on $X$ with bounded coherent cohomologies (see, for example, \cite{BLL}). One can check that $\Injdg(X)$ is pretriangulated and it is rather simple to see that there is a natural exact equivalence $\Ho(\Injdg(X))\iso\Db(X)$.

\subsubsection*{\v{C}ech resolutions}\label{subsubsect:Chech}

For $X$ as above, consider the pretriangulated dg category $\cat{C}(X)$ of bounded below complexes of quasi-coherent sheaves on $X$ with bounded and coherent cohomology. Clearly, $\Injdg(X)$ is a full dg subcategory of $\cat{C}(X)$. Let moreover $\cat{P}(X)$ be the full dg subcategory of $\cat{C}(X)$ whose objects are bounded complexes of locally free sheaves of finite rank.

We fix now a finite affine open cover $\mathcal{U}=\{U_1,\ldots U_r\}$ of $X$. Thus for any
\[
P:=\{\ldots\to 0\to P^n\to\ldots\to P^j\to\ldots\to P^{n+m}\to 0\ldots\}
\]
in $\cat{P}(X)$, with $m\geq 0$, we can take its \v{C}ech resolution $\check{C}_\mathcal{U}(P)$. Recall that this resolution is a complex of quasi-coherent sheaves which are coproducts of sheaves of the form
\[
P^j_U:=i_*i^*(P^j),
\]
where $U$ is the intersection of some of the open subsets in $\mathcal{U}$ and $i\colon U\mono X$ is the inclusion.

Define $\cat{P}(\mathcal{U})$ to be the smallest full pretriangulated subcategory of $\cat{C}(X)$ containing the \v{C}ech resolutions of all $P\in\cat{P}(X)$. Again, one can show that there is an exact equivalence $\Ho(\cat{P}(\mathcal{U}))\iso\Db(X)$ (see \cite[Lemma 5.6]{BLL}). Moreover, it is not too difficult to show that there is an isomorphism in $\Hqe$ between $\cat{P}(\mathcal{U})$ and $\Injdg(X)$.

This construction has been extended and considerably improved in \cite{LS}. The expert reader may also have a look at \cite{Sn} where the results in \cite{LS} have been strengthened.

\subsubsection*{Dolbeault enhancement}\label{subsubsect:Dol}

Let us illustrate a more analytic situation. For this, assume that $X$ is a complex manifold. We can then define a dg category $\cT_{\text{Dol}}(X)$ whose objects are the complex holomorphic vector bundles on $X$. Given two such vector bundles $E_1$ and $E_2$ we define the complex of morphism between them as the Dolbeault complex
\[
\cT_{\text{Dol}}(X)\left( E_1,E_2\right):= A^\bullet(E_1,E_2),
\]
where $A^q(E_1,E_2):=A^{0,q}(X,\mathcal{H}om(E_1,E_2))$ is the space of $(0,q)$-forms on $X$ with coefficients in the holomorphic bundle $\mathcal{H}om(E_1,E_2)$ of morphisms from $E_1$ to $E_2$. The differential
\[
\overline{\partial}\colon A^q(E_1,E_2)\to A^{q+1}(E_1,E_2)
\]
is defined by
\[
\overline{\partial}(\omega\otimes f):=\overline{\partial}(\omega)\otimes f+(-1)^q\omega\otimes\overline{\partial}(f)
\]
where $\overline{\partial}(f):=\overline{\partial}(f)_{E_2}\comp f-(f\otimes\id)\comp\overline{\partial}(f)_{E_1}$ and $\overline{\partial}_{E_i}$ is the natural operator induced by the holomorphic structure on $E_i$.

It turns out that there is a natural exact equivalence between $\Ho(\cT_{\text{Dol}}(X))$ and the exact category of holomorphic vector bundles on $X$ (see, for example, Example 9 in \cite[Section 2.3]{TLec} for more details). The reader can have a look at \autoref{subsect:prelfun} for the definition of exact category.

Of course, one would like to extract from this an enhancement of $\Db(X)$. One possibility would be to take finite locally free resolutions of coherent sheaves and try to imitate the argument above. Unfortunately, as explained in \cite{Bl}, this is problematic for several reasons. One of them is that such locally free resolutions do not always exist (see \cite{Vo}). Hence, to get an enhancement of $\Db(X)$, we need to take a slightly different perspective. Assume that $X$ is compact and let $A:=\cT_{\text{Dol}}(X)\left(\ko_X,\ko_X\right)$ be the dg algebra defined above. According to \cite[Definition 2.3.1]{Bl} one defines a dg category $\cP_A$ with the property that, by \cite[Theorem 4.1.3]{Bl}, we have a natural exact equivalence $\Ho(\cP_A)\iso\Db(X)$. It should be noted that here by $\Db(X)$ we mean the full subcategory of $\D(\Mod{\ko_X})$ consisting of complexes with bounded and coherent cohomology (which is not equivalent to $\Db(\Coh(X))$, in general).

The objects of $\cP_A$ are pairs $(E,\mathbb{E})$, where $E$ is a finitely generated and projective $\ZZ$-graded (right) module over $A^0=A^0(\ko_X,\ko_X)$ which is bounded above and below as a complex. The additional piece of data $\mathbb{E}$ is a $\CC$-linear map
\[
\mathbb{E}\colon E\otimes_{A^0} A\to E\otimes_{A^0} A
\]
of total degree $1$ such that
\[
\mathbb{E}(e\omega)=\mathbb{E}(e\otimes 1)\omega+(-1)^{\deg(e)}e\overline{\partial}(\omega)\qquad\text{and}\qquad \mathbb{E}\circ\mathbb{E}(e)=0,
\]
for all $e\in E$ and all $\omega\in A$. The reader is encouraged to have a look to \cite[Section 2.3]{Bl} for a detailed description of the morphisms in $\cP_A$.

\section{Well generated triangulated categories and localizations}\label{sect:wellgen}

In this section, we review some selected material concerning well generated categories. The focus is on Verdier quotients.

\subsection{Well generated triangulated categories}\label{subsect:introwell}

In this section we review some basic facts about well generated triangulated categories. The interested reader can find a very nice survey in \cite{K2}, where Neeman's original theory (see \cite{N2}) is revisited with the point of view of  \cite{K1}.

In this section, we assume $\cT$ to be a triangulated category with small coproducts. For a cardinal $\alpha$, an object $S$ of $\cT$ is \emph{$\alpha$-small} if every map $S\to\Plus_{i\in I}X_i$ in $\cT$ (where $I$ is a small set) factors through $\Plus_{i\in J}X_i$, for some $J\subseteq I$ with $\card{J}<\alpha$. We are then ready to state the main definition.

\begin{definition}\label{def:wellgen}
The category $\cT$ is \emph{well generated} if there exists a small set $\cS$ of objects in $\cT$ satisfying the following properties:
\begin{enumerate}
\item[(G1)]\label{G1} An object $X\in\cT$ is isomorphic to $0$, if and only if $\cT(S,X[j])=0$, for all $S\in\cS$ and all $j\in\ZZ$;
\item[(G2)]\label{G2} For every small set of maps $\{X_i\to Y_i\}_{i\in I}$ in $\cT$, the induced map $\cT(S,\Plus_iX_i)\to\cT(S,\Plus_i Y_i)$ is surjective for all $S\in\cS$, if $\cT(S,X_i)\to\cT(S, Y_i)$
is surjective, for all $i\in I$ and all $S\in\cS$;
\item[(G3)]\label{G3} There exists a regular cardinal $\alpha$ such that every object of $\cS$ is $\alpha$-small.
\end{enumerate}
\end{definition}

Recall that a cardinal $\alpha$ is called \emph{regular} if it is not the sum of fewer than $\alpha$ cardinals, all of them smaller than $\alpha$. A full triangulated subcategory of $\cT$ is \emph{$\alpha$-localizing} if it is
closed under $\alpha$-coproducts and under direct summands (the latter condition is automatic if $\alpha>\aleph_0$). By definition, an $\alpha$-coproduct is a coproduct of strictly less than $\alpha$ summands. Notice that a full triangulated subcategory of $\cT$ is localizing if and only if it is $\alpha$-localizing for every regular cardinal $\alpha$.

When the category $\cT$ is well generated and we want to put emphasis on the cardinal $\alpha$ in (G3) and on $\cS$, we say that $\cT$ is \emph{$\alpha$-well generated} by the set $\cS$. In this situation, following
\cite{K1}, we denote by $\cT^\alpha$ the smallest $\alpha$-localizing
subcategory of $\cT$ containing $\cS$. It is explained in \cite{K1,N2} that the category $\cT^\alpha$ does not depend on the choice of the set $\cS$ which well generates $\cT$.

The objects in $\cT^\alpha$ are called \emph{$\alpha$-compact}. Thus we will sometimes say that $\cT$ is
\emph{$\alpha$-compactly generated} by the set of \emph{$\alpha$-compact generators} $\cS$.
A very interesting case is when $\alpha=\aleph_0$. Indeed, with this choice, $\cT^\alpha=\cT^c$, the full
triangulated subcategory of compact objects in $\cT$. Recall that an object $C$ in $\cT$ is \emph{compact} if the functor $\cT(C,\farg)$ commutes with small coproducts. Notice that the compact objects in $\cT$ are precisely the $\aleph_0$-small ones.

It is very easy to see that $\cT$ is $\aleph_0$-compactly generated by a small set
$\cS\subseteq\cT^c$ if (G1) holds. Indeed (G2) and (G3) are automatically satisfied. In this situation, we simply say that $\cT$ is
\emph{compactly generated} by $\cS$. 

\begin{remark}\label{rmk:gen}
There is yet another notion of generation which is of interest in this paper. Indeed, given a small set $\cS$ of objects in $\cT$, we say that \emph{$\cS$ generates $\cT$} if $\cT$ is the smallest localizing subcategory of $\cT$ containing $\cS$. If $\cT$ is a well generated triangulated category, then a small set $\cS$ of objects in $\cT$ satisfies {\rm (G1)} if and only if $\cS$ generates $\cT$ (see, for example \cite[Proposition 5.1]{P}).
\end{remark}

\begin{remark}\label{rmk:G4}
Following \cite{N2}, we say that a small set of objects $\cS$ in a triangulated category $\cT$ with small coproducts satisfies (G4) if the following condition holds true.
\begin{enumerate}
\item[(G4)] For any small set of objects $\{X_i\}_{i\in I}$ in $\cT$ and any map $f\colon S\to\Plus_{i\in I} X$ with $S\in\cS$, there are objects $\{S_i\}_{i\in I}$ of $\cS$ and maps $f_i\colon S_i\to X_i$ and $g\colon S\to\Plus_{i\in I}S_i$ making the diagram
\[
\xymatrix{
S\ar[r]^-{f}\ar[dr]_{g}&\Plus_{i\in I}X_i\\
&\Plus_{i\in I}S_i.\ar[u]_-{\Plus f_i}
}
\]
commutative.
\end{enumerate}
By \cite[Lemma 4]{K1}, if $\cS$ is a small set of $\alpha$-small objects which is closed under $\alpha$-coproducts, for some regular cardinal $\alpha$, then $\cS$ satisfies (G2) if and only if it satisfies (G4). It is an easy exercise to show that (G4) implies (G2) in complete generality.

It is worth pointing out that (G1), (G3) and (G4) provide the original axiomatization for well generated triangulated categories given in \cite{N2}.
\end{remark}

\begin{ex}\label{ex:nowell}
Clearly, not all triangulated categories which are closed under small coproducts are well generated. This is the case of the homotopy category $\Kc(\Mod{\ZZ})$ of the abelian category $\Mod{\ZZ}$ of abelian groups. This is proved in \cite[Appendix E]{N2}.
\end{ex}

\subsection{Taking Verdier quotients}\label{subsect:wellquot}

Given a triangulated category $\cT$ with small coproducts and a localizing subcategory $\cL$ of $\cT$, the Verdier quotient $\cT/\cL$ might not have small Hom-sets. But if it has, then $\cT/\cL$ is a category (according to our convention) with small coproducts and the natural quotient functor
\[
\fQ\colon\cT\longrightarrow\cT/\cL
\]
commutes with small coproducts by \cite[Corollary 3.2.11]{N2}. Hence, in view of Theorem 5.1.1 and Proposition 2.3.1 in \cite{K2}, the exact functor $\fQ$ has a fully faithful right adjoint $\fQ^R$. Obviously, if $\cT/\cL$ is well generated then, by definition, it has small Hom-sets.

The behavior of well generation under taking Verdier quotients is clarified by the following result due to Neeman (see \cite{N2} but also \cite[Theorem 7.2.1]{K2}). It generalizes the seminal result \cite[Theorem 2.1]{N1}.

\begin{thm}\label{thm:compquot}
Let $\cT$ be a well generated triangulated category and let $\cL$ be a localizing subcategory which is generated by a small set of objects. Fix a regular cardinal $\alpha$ such that $\cT$ is $\alpha$-compactly generated and $\cL$ is generated by $\alpha$-compact objects. Then
\begin{itemize}
\item[{\rm (i)}] The localizing subcategory $\cL$ and the quotient $\cT/\cL$ are $\alpha$-compactly generated;
\item[{\rm (ii)}] $\cL^\alpha=\cL\cap\cT^\alpha$;
\item[{\rm (iii)}] The quotient functor $\fQ\colon\cT\to\cT/\cL$ sends $\cT^\alpha$ to $(\cT/\cL)^\alpha$;
\item[{\rm (iv)}] The induced functor $\fF\colon\cT^\alpha/\cL^\alpha\to(\cT/\cL)^\alpha$ if fully faithful and identifies $(\cT/\cL)^\alpha$ with the idempotent completion of $\cT^\alpha/\cL^\alpha$. If $\alpha>\aleph_0$, then $\fF$ is an equivalence.
\end{itemize}
\end{thm}

When we say that $(\cT/\cL)^\alpha$ is the idempotent completion of $\cT^\alpha/\cL^\alpha$ we mean that any object in $(\cT/\cL)^\alpha$ is isomorphic to a summand of an object in $\cT^\alpha/\cL^\alpha$.

\begin{ex}\label{ex:alphacompgen}
Let $\cG$ be a Grothendieck category. Recall that this means that $\cG$ is an abelian category which is closed under small coproducts, has a small set of generators $\cS$ and the direct limits of short exact sequences are exact in $\cG$. The objects in $\cS$ are generators in the sense that, for any $C$ in $\cG$, there exists an epimorphism $S\epi C$ in $\cG$, where $S$ is a small coproduct of objects in $\cS$. Without loss of generality, by taking the coproduct of all generators, we can assume that $\cG$ has a single generator $G$.

In this case, we can take $A:=\cG(G,G)$ to be the endomorphism ring of $G$. By \cite[Proposition 5.1]{AJ}  (see also \cite[Example 7.7]{K2}), there is an exact equivalence $\D(\cG)\iso\D(\Mod{A})/\cL$, for a localizing subcategory $\cL$ of $\D(\Mod{A})$. Here $\Mod{A}$ denotes the category of right $A$-modules. It was observed in \cite[Theorem 0.2]{N3} that $\D(\cG)$ is well generated.
\end{ex}

We can use this to show that there are triangulated categories which are well generated but not compactly generated. The following example is due to Neeman.

\begin{ex}[\cite{N3}]\label{ex:noncomp}
Let $X$ be a non-compact connected complex manifold of dimension greater than or equal to $1$. Consider the Grothendieck category $\Sh(X)$ of sheaves of abelian groups on $X$. The derived category $\D(\Sh(X))$ is well generated by \autoref{ex:alphacompgen}. But, on the other hand, by \cite[Theorem 0.1]{N3}, we have that $\D(\Sh(X))^c=\{0\}$ and thus $\D(\Sh(X))$ is not compactly generated.
\end{ex}

We want now to discuss a curious phenomenon that marks a difference between compactly generated and $\alpha$-compactly generated triangulated categories, for $\alpha>\aleph_0$. More precisely, we exhibit a triangulated category $\cT$ with a set $\cS$ of compact generators and a Verdier quotient $\fQ\colon\cT\to\cT/\cS$ such that $\fQ(\cS)$ satisfies (G1) and (G3) but not (G2).

Let $\cT$ be a well generated triangulated category and let $\cS$ be a set of $\alpha$-compact generators for $\cT$. Assume further that $\cL$ is a localizing subcategory of $\cT$ which is generated by a small set of $\alpha$-compact objects.

\begin{prop}\label{lem:inducedgen}
For $\cT$, $\cS$ and $\cL$ as above, the set $\fQ(\cS)$ satisfies {\rm (G1)} and {\rm (G3)} in $\cT/\cL$. In particular, if $\alpha=\aleph_0$, then $\fQ(\cS)$ compactly generates $\cT/\cL$.
\end{prop}

\begin{proof}
By \autoref{thm:compquot}, the Verdier quotient $\cT/\cL$ is $\alpha$-compactly generated. Hence, the fact that (G1) and (G3) hold true follows from \cite[Proposition 1.7]{CS6}. By the discussion in \autoref{subsect:introwell}, the set  $\fQ(\cS)$ satisfies (G2) as well, if $\alpha=\aleph_0$.
\end{proof}

On the other hand, if $\alpha>\aleph_0$, then we cannot expect that $\fQ(\cS)$ satisfies (G2). To see this, consider a commutative non-noetherian ring $R$ and take the Grothendieck category $\cC:=\Mod{R}$ of $R$-modules. By \autoref{ex:alphacompgen}, the triangulated category $\D(\cC)$ is well generated. 

By \cite{GU}, we have a filtration
\begin{equation}\label{eqn:filtr}
\cC=\bigcup_\alpha\cC^\alpha,
\end{equation}
where $\alpha$ runs over all sufficiently large regular cardinals and $\cC^\alpha$ is the full subcategory of $\cC$ consisting of $\alpha$-presentable objects. Recall that an object $A$ in $\cC$ is \emph{$\alpha$-presentable} if the functor $\cC(A,\farg)\colon\cC\to\Mod{\K}$ preserves $\alpha$-filtered colimits (see \cite[Section 6.4]{K2}, for the definition of $\alpha$-filtered colimit).

It is explained in the proof of \cite[Theorem 5.10]{K0} that $\D(\cC)$ can be realized as a Verdier quotient $\fQ\colon\D(\Mod{\cC^\alpha})\to\D(\cC)$ such that the kernel of $\fQ$ is generated by a small set of $\alpha$-compact objects, for all sufficiently large regular cardinals $\alpha$. Moreover, it is an easy exercise to show that $\D(\Mod{\cC^\alpha})$ is compactly generated by the set $\cS$ of objects in the image of $\Yon[\cC^{\alpha}]\colon\cC^\alpha\to\Mod{\cC^\alpha}\subset\D(\Mod{\cC^\alpha})$.

\begin{remark}\label{rmk:gc}
If $\fQ$ is the quotient functor mentioned above, then it is a general fact that the composition $\fQ\comp\Yon[\cC^{\alpha}]$ is isomorphic to the inclusion  $\cC^\alpha\subseteq\cC\subset\D(\cC)$ (see \cite[Corollary 5.3]{CS6}). Hence $\fQ(\cS)$ can be identified with $\cC^\alpha$.
\end{remark}

By \autoref{lem:inducedgen} the objects of $\fQ(\cS)$ satisfy (G1) and (G3). On the other hand, we have the following.

\begin{prop}\label{prop:inducedgen2}
In the above setting and for a sufficiently large regular cardinal $\alpha$, the objects in $\fQ(\cS)$ do not satisfy {\rm (G2)}.
\end{prop}

\begin{proof}
For a contradiction, assume that $\fQ(\cS)$ satisfies (G2). Let then $\{X_i\}_{i\in I}$ be a small set of objects in $\D(\cC)$ with the property that $\D(\cC)\left(S,X_i\right)=0$, for all $S\in\fQ(\cS)$ and all $i\in I$. Then we claim that $\D(\cC)\left(S,\Plus_i X_i\right)=0$, for all $S\in\cS$.

Indeed, it is obvious that the set of maps $\{0\to X_i\}_{i\in I}$ is such that the induced maps
\[
\D(\cC)\left(S,0\right)\longrightarrow\D(\cC)\left(S,X_i\right)\iso 0
\]
are surjective, for all $S\in\fQ(\cS)$. Since $\fQ(\cS)$ satisfies (G2), the induced map
\[
\D(\cC)\left(S,0\right)\longrightarrow\D(\cC)\Bigl(S,\Plus_i X_i\Bigr)
\]
is surjective for all $S\in\fQ(\cS)$ and so $\D(\cC)\left(S,\Plus_i X_i\right)=0$, for all $S\in\fQ(\cS)$.

Take now a small set $\{J_i\}_{i\in I}$ of injective $R$-modules whose direct sum $J$ is not injective. Note that such a set exists by
Bass--Papp Theorem (see \cite{Ch}) which states that every small coproduct of injective $R$-modules is injective if and only if $R$ is noetherian. As a consequence, there is an $R$-module $M$ such that $\D(\cC)(M,J[1])\neq 0$.

Let $\alpha$ be a large enough regular cardinal so
that $M\in\cC^\alpha$. Such an $\alpha$ exists because of \eqref{eqn:filtr}.
Then, for every $i\in I$ and every $N\in\cC^\alpha\subseteq\cC$, we have $\D(\cC)(N,J_i[1])=0$, because $J_i$ is an injective module. By the discussion above, and taking into account \autoref{rmk:gc}, this would imply $\D(\cC)(N,J[1])=0$, for all $N\in\cC^\alpha$. But $M\in\cC^\alpha$ and this is a contradiction.
\end{proof}

If $\cC$ is a Grothendieck category then there is a natural choice for a set of $\alpha$-compact generators. Indeed, if $\alpha$ is a sufficiently large regular cardinal, then, by \cite[Theorem 5.10]{K0}, we have a natural exact equivalence
\begin{equation}\label{eqn:alpcom}
\D(\cC)^\alpha\iso\D(\cC^\alpha)
\end{equation}
and $\cC^\alpha$ is an abelian category.
As an easy consequence of  \cite[Corollary 5.11]{K0}, the triangulated subcategory $\D(\cC)^\alpha$ forms a set of $\alpha$-compact generators for $\D(\cC)$. By \autoref{prop:inducedgen2}, the category $\D(\cC)^\alpha$ cannot be replaced by $\cC^\alpha$. Nevertheless, the following result shows that one can get a set of $\alpha$-compact generators by taking a full subcategory of $\D(\cC)^\alpha$.

\begin{prop}\label{prop:inducedgen3}
Let $\cC$ be a Grothendieck category and take a sufficiently large regular cardinal $\alpha$. Then $\Dm(\cC^\alpha)$ is a set of $\alpha$-compact generators for $\D(\cC)$.
\end{prop}

\begin{proof}
Since $\Dm(\cC^\alpha)$ is a full subcategory of $\D(\cC^\alpha)$ and the latter satisfy (G3), then the former satisfies (G3) as well. To prove (G1), let $X\in\D(\cC)$ be such that
\begin{equation}\label{eqn:zero}
\D(\cC)(A',X)=0,
\end{equation}
for all $A'\in\Dm(\cC^\alpha)$. On the other hand, if $A$ is any complex in $\D(\cC^\alpha)$, then, by \cite[Tags 093W and 0949]{SP}, we have an exact triangle
\[
\Plus_{i\in\ZZ}\tau_{\leq i} A\to\Plus_{i\in\ZZ}\tau_{\leq i} A\to A.
\]
Since for $\alpha$ sufficiently large $\cC^\alpha$ is an abelian category, we have $\tau_{\leq i} A\in\Dm(\cC^\alpha)$ and
\[
\D(\cC)\Bigl(\Plus_{i\in\ZZ}\tau_{\leq i} A, X\Bigr)\iso\prod_{i\in\ZZ}\D(\cC)\Bigl(\tau_{\leq i} A, X\Bigr)\iso 0,
\]
by \eqref{eqn:zero}. Hence $\D(\cC)(A, X)\iso0$, for all $A\in\D(\cC^\alpha)$. Since we observed above that the objects in $\D(\cC^\alpha)$ satisfy (G1), we conclude that $X\iso0$.

We know that $\D(\cC^\alpha)$ satisfies (G2) and thus, by \autoref{rmk:G4}, it satisfies (G4). Take a small set of objects $\{X_i\}_{i\in I}$ in $\cT$ and a map $f\colon S\to\Plus_{i\in I} X_i$ with $S\in\Dm(\cC^\alpha)$. Clearly, there are objects $\{S_i\}_{i\in I}$ of $\D(\cC^\alpha)$ and maps $f_i\colon S_i\to X_i$ making the diagram
\[
\xymatrix{
S\ar[r]^-{f}\ar[dr]_{g}&\Plus_{i\in I}X_i\\
&\Plus_{i\in I}S_i.\ar[u]_-{\Plus f_i}
}
\]
commutative, for some $g\colon S\to\Plus_{i\in I}S_i$. Since $S$ is bounded above, the morphism $g$ factors through $\tau_{\leq n}\left(\Plus_{i\in I}S_i\right)\iso\Plus_{i\in I}\tau_{\leq n}(S_i)$, for some $n\in\ZZ$.  This implies that $\Dm(\cC^\alpha)$ satisfies (G4). Again by \autoref{rmk:G4}, the triangulated category $\Dm(\cC^\alpha)$ satisfies (G2).
\end{proof}

\section{Bad news: enhancing triangulated categories}\label{sect:badcat}

In this section we explain the known examples of triangulated categories without enhancements or with non-unique enhancements. In the algebro-geometric setting it will be explained later that the quite disheartening scenery pictured in this section cannot appear.

\subsection{Preliminaries: algebraic triangulated categories}\label{subsect:prelfun}

According to \cite{Qu}, an \emph{exact category} is an additive category $\cA$ with a distinguished class of sequences
\begin{equation}\label{exseq}
\xymatrix{
0\ar[r]&A\ar[r]^-{i}&B\ar[r]^-{p}&C\ar[r]&0,} 
\end{equation}
which are called \emph{exact} and satisfy some axioms which we do not need to list here. In particular $i$ is a kernel of $p$ and $p$ is a cokernel of $i$. For example, every abelian category, and more generally every full and extension-closed subcategory of an abelian category, is an exact category with respect to the class of short exact sequences. In fact it can be proved that every small exact category is of this form, up to equivalence (see, for example, \cite[Section 7.7]{K3}). Many common constructions of homological algebra can be extended from abelian to exact categories. For instance, an object $A$ in an exact category $\cA$ is \emph{injective} (respectively \emph{projective}) if the functor $\cA(\farg,A)$ (respectively $\cA(A,\farg)$) preserves exact sequences. Then one says that $\cA$ has \emph{enough injectives} (respectively \emph{enough projectives}) if for every object $A$ (respectively $C$) of $\cA$ there is an exact sequence as in \eqref{exseq} such that $B$ is injective (respectively projective).

A \emph{Frobenius category} is an exact category with enough injectives, enough projectives and such that an object is injective if and only if it is projective. The \emph{stable category} $\St{\cF}$ of a Frobenius category $\cF$ has the same objects as $\cF$, whereas morphisms are stable equivalence classes of morphisms of $\cF$. Here two morphisms $f,g\colon A\to B$ of $\cF$ are defined to be stably equivalent if $f-g$ factors through an injective (or projective) object. There is a natural way to put a triangulated structure on $\St{\cF}$. Indeed, the shift is defined by choosing, for every object $A$, an exact sequence
\[
\xymatrix{
0 \ar[r] & A \ar[r] & I \ar[r] & A[1] \ar[r] & 0
}
\]
in $\cF$ with $I$ injective. Moreover, the distinguished triangles are those isomorphic to the images in $\St{\cF}$ of the sequences in $\cF$ of the form $(i,p,e)$, where $(i,p)$ defines an exact sequence as in \eqref{exseq} and $e$ is part of a commutative diagram
\[
\xymatrix{
0 \ar[r] & A \ar[d]_{\id} \ar[r]^-i & B \ar[d] \ar[r]^-p & C \ar[d]^e \ar[r] & 0 \\
0 \ar[r] & A \ar[r] & I \ar[r] & A[1] \ar[r] & 0.
} 
\]

\begin{prop}\label{algchar}
The following conditions are equivalent for a triangulated category $\cT$.
\begin{enumerate}
\item\label{condsub} There is an exact and fully faithful functor $\cT\to\St{\cF}$ for some Frobenius category $\cF$.
\item\label{condFrob} There is an exact equivalence $\cT\to\St{\cF}$ for some Frobenius category $\cF$.
\item\label{condhomot} There is an exact and fully faithful functor $\cT\to\Kc(\cA)$ for some additive category $\cA$.
\item\label{condenhance} $\cT$ is algebraic.
\end{enumerate}
\end{prop}

\begin{proof}
If $\cF$ is a Frobenius category and $\fF\colon\cT\to\St{\cF}$ is an exact and fully faithful functor, let $\cF'$ be the full subcategory of $\cF$ whose objects are those in the essential image of $\fF$ (when regarded as objects of $\St{\cF}$). It is easy to see that also $\cF'$ is a Frobenius category, hence \eqref{condsub} implies \eqref{condFrob}.
  
On the other hand, let $\Kai(\cF)$ be the full subcategory of $\Kc(\cF)$ whose objects are acyclic complexes with injective components. It is not difficult to prove that there is an exact equivalence $\Kai(\cF)\to\St{\cF}$ defined on objects by $C\mapsto Z^0(C)$. This shows that \eqref{condFrob} implies \eqref{condhomot}.

As $\Cdg(\cA)$ is an enhancement of $\Kc(\cA)$ for every additive category $\cA$ by \autoref{ex:derenh} (iii), it is clear that \eqref{condhomot} implies \eqref{condenhance}.

If $\cC$ is an enhancement of $\cT$, then $\Ho(\dgYon[\cC])\colon\Ho(\cC)\to\Ho(\dgMod{\cC})$ is an exact and fully faithful functor. Hence the fact that \eqref{condenhance} implies \eqref{condsub} follows easily from \cite[Lemma 2.2]{K}, where it is proved that $\Zo(\dgMod{\cC})$ is a Frobenius category such that $\St{\Zo(\dgMod{\cC})}=\Ho(\dgMod{\cC})$.
\end{proof}

\begin{remark}\label{rmk:loc1}
It is clear that any full triangulated subcategory $\cS$ of an algebraic triangulated category $\cT$ is algebraic. If the quotient $\cT/\cS$ is a category in our sense, then, by \autoref{ex:derenh} (ii), it is algebraic as well.
\end{remark}

The following concerns a sort of converse of the above remark.

\begin{qn}\label{qn:loc1}
Let $\cS$ be a full triangulated subcategory of a triangulated category $\cT$. Assume that $\cS$ and $\cT/\cS$ are algebraic. Is $\cT$ algebraic as well?
\end{qn}

\begin{ex}\label{ex:Grothalg}
If $\cA$ is an abelian category such that $\D(\cA)$ is a category, then, by the discussion in \autoref{subsect:exenh} (see also \autoref{rmk:loc1}), the triangulated category $\D(\cA)$ is algebraic. One instance where this is certainly true is when $\cA$ is a Grothendieck category. Indeed, in this situation, $\D(\cA)$ is well generated by \autoref{ex:alphacompgen}. On the other hand, in this case there is an exact equivalence between $\D(\cA)$ and the full subcategory of $\Kc(\cA)$ with objects the h-injective complexes (see \cite{Se}). Recall that $I\in\Kc(\cA)$ is called h-injective (or K-injective, as in \cite{Se}) if $\Kc(\cA)(A,I)=0$ for every acyclic complex $A\in\Kc(\cA)$. It follows that, when $\cA$ is a Grothendieck category, a more explicit enhancement of $\D(\cA)$ is given by the full dg subcategory of $\Cdg(\cA)$ with objects the h-injective complexes. Clearly, the full subcategory $\D(\cA)^c$ of $\D(\cA)$ is algebraic as well.

This applies, in particular, when $\cA=\Qcoh(X)$ and $X$ is an algebraic stack. Under additional assumptions on $X$ (see \autoref{cor:geocomp}), $\D(\Qcoh(X))^c$ is the category $\Dp(X)$ of perfect complexes on $X$.
\end{ex}

\subsection{Counterexample to the existence of dg enhancements}\label{subsect:nonexistence}

We start by highlighting the following construction inspired by algebraic topology. A \emph{spectrum} is a sequence $E$ of pointed topological spaces $E_n$ and pointed continuous maps $\sigma_n\colon\Sigma E_n\to E_{n+1}$, called structural morphisms, with an additional property which we describe below. For a topological space $X$ with basepoint $x_0$, $\Sigma X$ stands for the \emph{reduced suspension} of $X$, i.e. the pointed topological space obtained from $X\times [0,1]$ by collapsing $X\times\{0\}$ and $X\times\{1\}$ to two different points and further collapsing $\{x_0\}\times[0,1]$ to a point which is then the basepoint of $\Sigma X$. More precisely, the reduced suspension defines a functor $\Sigma$ in the category of topological spaces with a right adjoint $\Omega$. Indeed, for a topological space $X$, we denote by $\Omega X$ the loop space of $X$. We then require that $\sigma_n$ becomes a homeomorphism under this adjunction. 

\begin{ex}\label{ex:spherespectrum}
If we take $E_n$ to be the $n$-dimensional (pointed) sphere $S^n$ with the obvious identification map $\Sigma S^n\to S^{n+1}$, we get the so called \emph{sphere spectrum}.
\end{ex}

A morphism of spectra $E\to F$ is a sequence of pointed continuous maps $E_n\to F_n$ strictly compatible with the structural homeomorphisms. Denote by $\Spe$ the category of spectra.

The homotopy groups of a spectrum $E$ are the groups
\[
\pi_n(E):=\colim{i\geq\max\{0,-n\}}\pi_{n+i}(E_i),
\]
where the colimit is computed with respect to the natural maps
\[
\pi_{n+i}(E_i)\to\pi_{n+i+1}(\Sigma E_i)\iso\pi_{n+i+1}(E_{i+1}).
\]
Now, one can check that the above composition of morphisms is actually an isomorphism, for $i\geq 0$ and $n+i\geq 0$ (see \cite[Definition 10.9.4]{We}). In other words, we have a natural isomorphism $\pi_n(E)\iso\pi_{n+i}(E_i)$, for $i\geq 0$ and $n+i\geq 0$.

A morphism between spectra is a \emph{weak equivalence} if it induces an isomorphism on the homotopy groups. As it is customary, the \emph{stable homotopy category} $\ho{\Spe}$ is obtained from the category $\Spe$  by formally inverting the weak equivalences. Notice that $\Spe$ can be seen as a model category and $\ho{\Spe}$ is its homotopy category (see \autoref{rmk:modelcat}). It is well-known that $\ho{\Spe}$ is a $\ZZ$-linear triangulated category (see \cite{A,Pu}). 

The following argument is taken from \cite[Section 7.6]{K3}.

\begin{prop}\label{prop:noalg}
The triangulated category $\ho{\Spe}$ is not algebraic.
\end{prop}

\begin{proof}
By \autoref{algchar} it is enough to show that there is no exact and fully faithful functor
\[
\fF\colon\ho{\Spe}\to\Kc(\cA),
\]
where $\cA$ is an additive category.

Suppose that such a functor exists and let $S\in\ho{\Spe}$ be the sphere spectrum defined in \autoref{ex:spherespectrum}. Let moreover $S'$ be an object sitting in a distinguished triangle
\[
\xymatrix{
S\ar[r]^-{2\cdot\id_S}& S\ar[r]&S'
}
\]
in $\ho{\Spe}$. It is possible to show, using the Steenrod operations, that $2\cdot\id_{S'}\neq 0$ (see, for example, \cite{S2}).

On the other hand, let $X:=F(S)$ and $X':=\fF(S')$. As $\fF$ is exact, $X'$ is isomorphic in $\Kc(\cA)$ to the mapping cone $X''$ of $2\cdot\id_{X}$. It is now a simple exercise to show that the induced morphism $2\cdot\id_{X''}$ of complexes is null-homotopic. Hence $2\cdot\id_{X'}=0$ in $\Kc(\cA)$. This contradicts the fact that $\fF$ is faithful.
\end{proof}

\begin{remark}\label{rmk:nomodel}
A triangulated category $\cT$ is \emph{topological} if there is an exact equivalence between $\cT$ and the homotopy category of a stable cofibration category $\cC$ (see \cite{S1} for a detailed discussion). In this case $\cC$ can be regarded as a sort of `topological enhancement' of $\cT$. It is important to observe that every algebraic triangulated category is topological (see \cite[Proposition 3.2]{Sw}), but the converse is not true. Indeed, as the reader may suspect, the triangulated category $\ho{\Spe}$ is topological. On the other hand, the main result of \cite{MSS} shows that there is a $\ZZ$-linear triangulated category which is not topological.
\end{remark}

Nevertheless, the following is still a challenging problem.

\begin{qn}\label{qn:exfield}
Is there a triangulated category, linear over a field $\K$, which is not algebraic?
\end{qn}

\subsection{Counterexample to the uniqueness of dg enhancements}\label{subsect:nonuniqueness}

We assume $\K=\ZZ$. If $R$ is a quasi-Frobenius ring (meaning that $\Mod{R}$ is a Frobenius category), then $\Mod{R}$ has a model structure defined as follows (see \cite{DS} or \cite[Section 2.2]{Ho}): cofibrations are monomorphisms, fibrations are epimorphisms and weak equivalences are morphisms which become isomorphisms in $\St{\Mod{R}}$. Moreover, the homotopy category $\ho{\Mod{R}}$ coincides with $\St{\Mod{R}}$.

Let $p$ be a prime number and $\FF_p=\ZZ/p\ZZ$ the field with $p$ elements. It is well known that $R_1:=\ZZ/p^2\ZZ$ and $R_2:=\FF_p[\ep]$ (where $\ep^2=0$) are quasi-Frobenius rings (in both cases a module is injective if and only if it is projective if and only if it is free). We denote by $\cTr$ the category $\Mod{\FF_p}$ endowed with the triangulated structure defined by $[1]=\id$ (and, necessarily, distinguished triangles given by triangles inducing long exact sequences). We recall the following important result (see \cite{Sc}, \cite{DS}).

\begin{thm}\label{thm:DS}
The model categories $\Mod{R_1}$ and $\Mod{R_2}$ are not Quillen equivalent, but there is an exact equivalence between $\ho{\Mod{R_i}}$ and $\cTr$, for $i=1,2$.
\end{thm}

Recall that a \emph{Quillen adjunction} between two model categories $\cC_1$ and $\cC_2$ is given by two adjoint functors $\adj{\fF}{\fG}{\cC_1}{\cC_2}$ such that $\fF$ preserves cofibrations and trivial cofibrations (or, equivalently, $\fG$ preserves fibrations and trivial fibrations). By definition, a trivial (co)fibration is a morphism which is both a (co)fibration and a weak equivalence. A Quillen adjunction as above induces two adjoint functors $\adj{\ld\fF}{\rd\fG}{\ho{\cC_1}}{\ho{\cC_2}}$ between the homotopy categories, and it is called a \emph{Quillen equivalence} if $\ld\fF$ (or $\rd\fG$) is an equivalence of categories. Finally, two model categories are \emph{Quillen equivalent} if they are related by a zigzag of Quillen equivalences.

Back to the statement of \autoref{thm:DS}, it is easy to see that there are exact equivalences between the homotopy categories and $\cTr$. On the other hand, the fact that the two model categories are not Quillen equivalent is a deep result, whose original proof in \cite{Sc} uses rather involved $K$-theoretical computations. Also the simpler proof in \cite{DS} is far from trivial and uses homotopy endomorphism ring spectra.

Denoting by $\Cdgai(\Mod{R_i})$ the full dg subcategory of $\Cdg(\Mod{R_i})$ with objects the acyclic complexes with injective components (for $i=1,2$), we obtain the following counterexample to the uniqueness of enhancements (mentioned in \cite[Remark 2.3]{S}).

\begin{cor}\label{nonuni}
The pretriangulated dg categories $\Cdgai(\Mod{R_1})$ and $\Cdgai(\Mod{R_2})$ are not isomorphic in $\Hqe$, but both of them are enhancements of $\cTr$.
\end{cor}

\begin{proof}
For $i=1,2$, $\Cdgai(\Mod{R_i})$ is naturally an enhancement of $\Kai(\Mod{R_i})$. Since there are exact equivalences $\Kai(\Mod{R_i})\isomor\St{\Mod{R_i}}$ (see the proof of \autoref{algchar}) and $\St{\Mod{R_i}}\iso\ho{\Mod{R_i}}\isomor\cTr$ (by \autoref{thm:DS}), it follows that $\Cdgai(\Mod{R_i})$ is an enhancement of $\cTr$, for $i=1,2$.

Assume now that $\Cdgai(\Mod{R_1})\iso\Cdgai(\Mod{R_2})$ in $\Hqe$. Let $C_i\in\Cdgai(\Mod{R_i})$ be the complex of $R_i$-modules with $R_i$ in each degree and whose differential is multiplication by $p$ for $i=1$ and by $\varepsilon$ for $i=2$. It is not difficult to show that any equivalence between $\Kai(\Mod{R_i})$ and $\cTr$ must send $C_i$ to $\FF_p$, up to isomorphism (this can be seen, for instance, by checking that $\Kai(\Mod{R_i})(C_i,C_i)$ is a one-dimensional $\FF_p$-vector space). Hence $C_1$ corresponds to $C_2$, up to isomorphism, in the equivalence between $\Kai(\Mod{R_1})$ and $\Kai(\Mod{R_2})$ induced by the given isomorphism in $\Hqe$. Denoting by $\cC_i$ the full dg subcategory of $\Cdgai(\Mod{R_i})$ whose only object is $C_i$ (so that $\cC_i$ is actually a dg algebra), it follows that $\cC_1\iso\cC_2$ in $\Hqe$. This implies that $\Zo(\dgMod{\cC_1})$ and $\Zo(\dgMod{\cC_2})$ with their standard model structures (namely, weak equivalences are quasi-isomorphisms and fibrations are surjections) are Quillen equivalent (see, for instance, \cite[Corollary 3.4]{To}). On the other hand, by \cite[Theorem 3.5]{DS}, for $i=1,2$ there is a Quillen equivalence between $\Zo(\dgMod{\cC_i})$ and $\Mod{R_i}$ (taking into account that the complex of $R_i$-modules denoted by $P_{\bullet}M$ in \cite{DS} is just $C_i$ when $M=\FF_p$ as $R_i$-module). Thus we conclude that $\Mod{R_1}$ and $\Mod{R_2}$ are Quillen equivalent, contradicting \autoref{thm:DS}. Therefore $\Cdgai(\Mod{R_1})\not\iso\Cdgai(\Mod{R_2})$ in $\Hqe$.
\end{proof}

\begin{remark}\label{nonuniquot}
The above argument can be clearly adapted to prove that also the full (triangulated) subcategory $\cTrc$ of $\cTr$ whose objects are finite dimensional $\FF_p$-vector spaces does not admit a unique enhancement (the two different enhancements are of course given by the full dg subcategories of $\Cdgai(\Mod{R_i})$ whose objects correspond to those of $\cTrc$ under the equivalence between $\Kai(\Mod{R_i})$ and $\cTr$). Setting $X:=\spec(R_2)$, it is easy to see that there is an exact equivalence between $\cTrc$ and $\Db(X)/\Dp(X)$. This is easy to show using the classification of indecomposable objects in $\Db(X)$ (see \cite{AM} and \autoref{rmk:students} for the definition of indecomposable objects). As $\Db(X)$ and $\Dp(X)$ have a unique enhancement (see \autoref{rmk:coh} and \autoref{cor:geocomp}, respectively), this example proves that the property of admitting a unique enhancement is not stable under passage to Verdier quotients.
\end{remark}

\begin{remark}\label{nonunifield}
The category $\cTr$ is not only $\ZZ$-linear, but also $\FF_p$-linear. Clearly the same is true for $\Cdgai(\Mod{R_2})$, but not for $\Cdgai(\Mod{R_1})$, and for this reason we had to take $\K=\ZZ$.
\end{remark}

As in \autoref{qn:exfield}, it makes then sense to formulate the following.

\begin{qn}\label{qn:uniqenessfield}
Are there examples of triangulated categories, linear over a field $\K$, with non-unique $\K$-linear enhancement?
\end{qn}

Let $\cS$ be a full triangulated subcategory of a triangulated category $\cT$ . We can then formulate the following additional question.

\begin{qn}\label{qn:uniqloc}
Does $\cS$ have unique enhancements if $\cT$ does? What if $\cS$ is a localizing (in case $\cT$ has small coproducts) or an admissible subcategory of $\cT$?
\end{qn}

Recall that a full triangulated subcategory $\cS$ of a triangulated category $\cT$ is \emph{admissible} if the inclusion functor $\cS\mono\cT$ has right and left adjoints.

\begin{ex}\label{ex:curves}
There is a rather simple series of geometric examples where the question above can be answered positively, in the case of admissible subcategories. Indeed, let $C$ be a smooth projective curve over an algebraically closed field $\K$ and let $\Db(C)$ be the bounded derived category of coherent sheaves on $C$. Let $\cS$ be an admissible subcategory of $\Db(C)$.

It is known (see, for example, the proof of \cite[Proposition 6.11]{CS3}) that, if $\cS$ is not trivial or the whole $\Db(C)$, then $C=\PP^1$ and there is an exact equivalence $\cS\iso\Db(\spec(\K))$. As a simple consequence of what we will prove later (see \autoref{cor:geocomp}), we conclude that $\cS$ has a unique enhancement.
\end{ex}

\section{Uniqueness of enhancements in the unbounded case}\label{sect:uniqcat}

In this section we start analyzing the positive side of the story concerning the enhancements of triangulated categories. In particular, we focus on the enhancements of the unbounded derived category of a Grothendieck category. The result in this context follows almost directly from a rather general criterion which we discuss first. We spend some time explaining some geometric applications and some open problems.

\subsection{The statement}\label{subsect:statemunbounded}

All the results about the uniqueness of enhancements for the unbounded derived categories which will be discussed in this paper are simple consequences of a rather abstract result. But for this, we need the following.

\begin{definition}\label{abc}
	Let $\cT$ be a triangulated category with small coproducts. An exact functor $\fF\colon\dgD(\cA)\to\cT$ is \emph{right vanishing} if it preserves small coproducts and there exists a full subcategory $\cR$ of $\cT$ with the following properties:
	\begin{enumerate}
		\item[(R1)]\label{Rcoprod} $\cR$ is closed under small coproducts;
		\item[(R2)]\label{Rext} $\cR$ is closed under extensions (meaning that, if $X\to Y\to Z$ is a distinguished triangle in $\cT$ with $X,Z\in\cR$, then $Y\in\cR$, as well); 
		\item[(R3)]\label{Rrep} $\fF(\Yon(A))[k]\in\cR$ for every $A\in\cA$ and every integer $k<0$;
		\item[(R4)]\label{Rort} $\cT\left(\fF(\Yon(A)),R\right)=0$ for every $A\in\cA$ and every $R\in\cR$.
	\end{enumerate}
\end{definition}

Here we can regard $\Yon$ as a functor $\cA\to\dgD(\cA)$ thanks to the natural identification $\D(\Mod{\cA})\iso\dgD(\cA)$. The statement is as follows.

\begin{thm}[\cite{CS6,CS7}, Theorem C]\label{thm:crit1}
Let $\cA$ be a small category which we consider as a dg category concentrated in degree zero and let $\cL$ be a localizing subcategory of $\dgD(\cA)$ such that:
\begin{itemize}
\item[{\rm (a.1)}] The quotient $\dgD(\cA)/\cL$ is a well generated triangulated category;
\item[{\rm (b.1)}] The quotient functor $\fQ\colon\dgD(\cA)\to\dgD(\cA)/\cL$ is right vanishing.
\end{itemize}
Then $\dgD(\cA)/\cL$ has a unique enhancement.
\end{thm}

\autoref{thm:crit1} appeared in \cite{CS6} with (b.1) replaced by the weaker assumption
\[
\dgD(\cA)/\cL\left(\fQ(\Yon(A)),\coprod_{i\in I}\fQ(\Yon(A_i))[k_i]\right)=0,
\]
for all $A,A_i\in\cA$ (with $I$ a small set) and all integers $k_i<0$. The original proof contains a gap and the statement has been changed to the one above in \cite{CS7} where a complete proof is provided. It should be noted that the statement of \autoref{thm:crit1} and its proof (both the old one in \cite{CS6} and the new one in \cite{CS7}) are very much inspired by the following weaker result due to Lunts and Orlov.

\begin{thm}[\cite{LO}, Theorem 1]\label{thm:LO1}
Let $\cA$ be a small category which we consider as a dg category concentrated in degree zero and let $\cL$ be a localizing subcategory of $\dgD(\cA)$ such that the Verdier quotient $\dgD(\cA)/\cL$ is a category and:
\begin{itemize}
\item[{\rm (a.2)}] The functor $\fQ$ sends $\dgD(\cA)^c$ to $(\dgD(\cA)/\cL)^c$;
\item[{\rm (b.2)}] $\dgD(\cA)/\cL\left(\fQ(\Yon(A)),\fQ(\Yon(A'))[k]\right)=0$, for all $A,A'\in\cA$ and all integers $k<0$.
\end{itemize}
Then $\dgD(\cA)/\cL$ has a unique enhancement.
\end{thm}

A comparison between the two results is in order here.

\begin{lem}\label{lem:compar1}
Let $\cA$ be a small category and let $\cL$ be a localizing subcategory of $\dgD(\cA)$ such that the Verdier quotient $\dgD(\cA)/\cL$ is a category and {\rm (a.2)} holds. Then $\dgD(\cA)/\cL$ is compactly generated and so {\rm (a.1)} holds true.
\end{lem}

\begin{proof}
The objects in $\Yon(\cA)$ form a set of compact generators for $\dgD(\cA)$. Thus, we just need to show that, under the hypotheses, the category $\dgD(\cA)/\cL$ is compactly generated by the compact objects in $\fQ\comp\Yon(\cA)$. For this, assume that $B$ in $\dgD(\cA)/\cL$ is such that $\dgD(\cA)/\cL(\fQ(\Yon(A)),B[j])=0$, for all $A\in\cA$ and all integers $j\in\ZZ$. As $\fQ$ has a fully faithful right adjoint $\fQ^R$ (see the discussion at the beginning of \autoref{subsect:wellquot}), we have
\[
\dgD(\cA)/\cL(\fQ(\Yon(A)),B[j])\iso\dgD(\cA)(\Yon(A),\fQ^R(B)[j])=0,
\]
for all $A\in\cA$ and all integers $j\in\ZZ$. Since $\Yon(\cA)$ is a set of compact generators for $\dgD(\cA)$, we have $\fQ^R(B)\iso 0$. Hence $B\iso 0$.
\end{proof}

Let us now observe that (a.2) and (b.2) together imply (b.1). Indeed, it is easy to see that $\fQ$ is right vanishing, taking $\cR$ to be the full subcategory of $\dgD(\cA)/\cL$  consisting of those $R$ such that $\dgD(\cA)/\cL\left(\fQ(\Yon(A)),R\right)=0$ for every $A\in\cA$. Hence \autoref{thm:crit1} implies \autoref{thm:LO1}.

To conclude that the first result really improves the second one, we just need to exhibit a triangulated category of the form  $\dgD(\cA)/\cL$, with $\cA$ and $\cL$ as above, which satisfies (a.1) and (b.1) but not (a.2).

To this end, consider the triangulated category $\D(\cG)$, with $\cG$ a Grothendieck category. We know that there is an exact equivalence
\begin{equation}\label{eqn:equiv}
\D(\cG)\iso\D(\Mod{A})/\cL,
\end{equation}
where $A$ is the endomorphism ring of a generator $G$ of $\cG$, and that the quotient category is well generated (see \autoref{ex:alphacompgen}). As $\D(\Mod{A})\iso\dgD(A)$, where we think of $A$ as a category with one object, we see that (a.1) holds. As explained in \autoref{rmk:gc}, up to isomorphism, $\fQ(\Yon(A))$ is sent precisely to $G$ under \eqref{eqn:equiv}. Thus, since there are no negative Hom's between objects of $\cG$, (b.2) holds. Moreover, take $\cR$ to be the full subcategory of $\dgD(\cA)/\cL\cong\D(\Mod{A})/\cL$ whose objects are sent, under the equivalence \eqref{eqn:equiv}, to objects of $\D(\cG)$ with cohomologies in strictly positive degree. With this choice, $\fQ$ is right vanishing and (b.1) holds true. On the other hand, (a.2) does not hold, by \autoref{lem:compar1}, whenever $\D(\cG)$ is not compactly generated (for instance, when $\cG=\Sh(X)$ as in \autoref{ex:noncomp}).

\subsection{The idea of the proof of \autoref{thm:crit1}}\label{subsect:proofunbounded}

Roughly speaking, the proof of \autoref{thm:crit1} consists of two main steps, which are very much inspired by the approach in \cite{LO}, with some major changes. The first one is of pure dg flavor and makes use both of (a.1) and (b.1). The second one, much more technical, has a triangulated nature and uses the hypothesis (b.1). Here we will mainly concentrate on the first step.

\subsubsection*{Defining the functor} Let $\fE\colon\dgD(\cA)/\cL\to\Ho(\cC)$ be an exact equivalence, for some pretriangulated dg category $\cC$ and take the composition
\[
\xymatrix{
\fH\colon\cA\ar[r]^-{\Yon}&\dgD(\cA)\ar[r]^-{\fQ}&\dgD(\cA)/\cL\ar[r]^-{\fE}&\Ho(\cC)
}
\]
of functors. Let us denote by $\cB_0$ the full dg subcategory of $\cC$ such that the objects of $\cB_0$ are those in the image of $\fH$. Hence we can clearly regard $\fH$ as a dg functor $\cA\to\Ho(\cB_0)$, where $\cA$ and $\Ho(\cB_0)$ are thought of as dg categories concentrated in degree zero.

\subsubsection*{The truncation} Let $\tau_{\leq 0}(\cB_0)$ be the dg category with the same objects as $\cB_0$, while
\[
\tau_{\le0}(\cB_0)\left(B_1,B_2\right):=\tau_{\le0}\left(\cB_0\left(B_1,B_2\right)\right).
\]
for every $B_1$ and $B_2$ in $\cB_0$. There are obvious dg functors $\tau_{\le0}(\cB_0)\to\Ho(\cB_0)$ and $\tau_{\le0}(\cB_0)\to\cB_0$.

\begin{lem}\label{lem:tau}
The dg functor $\tau_{\le0}(\cB_0)\to\Ho(\cB_0)$ is a quasi-equivalence.
\end{lem}

\begin{proof}
It is enough to show that $\Ho(\cC)\left(B_1,B_2[j]\right)=0$, for all $B_1,B_2\in\cB_0$ and all $j<0$. This follows from (b.1) and the fact that $\fE$ is an equivalence.
\end{proof}

In conclusion, we get a quasi-functor $\Ho(\cB_0)\to\cB_0$.

\subsubsection*{Using that $\dgD(\cA)/\cL$ is well generated} Since $\fE$ is an exact equivalence and we assume (a.1), we have that $\Ho(\cC)$ is a well generated triangulated category. By \cite[Theorem 2.8]{CS6}, there exist a regular cardinal $\alpha$ and a small and full dg subcategory $\cB$ of $\cC$ containing $\cB_0$ such that $\Ho(\cB)$ is closed under $\alpha$-coproducts and an exact equivalence 
\begin{equation}\label{eqn:al}
\fPo\colon\Ho(\cC)\to\dgD_\alpha(\cB)
\end{equation}
which is induced by the natural quasi-functor $\cC\to\hproj{\cB}$ (see \autoref{ResSF}), whose image is actually contained in the enhancement $\hproja{\cB}$ of $\dgD_\alpha(\cB)$.

Let us give a bit more details about this. Recall that, following \cite{P}, we denote by $\dgD_\alpha(\cB)$ the
\emph{$\alpha$-continuous derived category} of $\cB$. By definition, it is
the full subcategory of $\dgD(\cB)$ with objects those $M\in\dgMod{\cB}$ such that the natural map
\[
(H^*(M))\Bigl(\Plus_{i\in I}C_i\Bigr)\longrightarrow\prod_{i\in I}(H^*(M))(C_i)
\]
(where the coproduct is intended in $\Ho(\cB)$) is an isomorphism, for all objects $C_i\in\cB$, with $\card{I}<\alpha$.
Clearly, $\dgD_\alpha(\cB)$ has an obvious enhancement $\hproja{\cB}$ given as the full dg subcategory of $\hproj{\cB}$ whose objects correspond to those in  $\dgD_\alpha(\cB)$, under the equivalence $\Ho(\hproj{\cB})\iso\dgD(\cB)$. Analyzing more carefully \cite{P} (see \cite[Section 2.2]{CS6}), one sees that there is a natural quasi-functor
\begin{equation}\label{eqn:alphaquot}
\hproj{\cB}\to\hproja{\cB}.
\end{equation}

In conclusion, we compose $\fH\colon\cA\to\Ho(\cB_0)$ with the quasi-functor $\Ho(\cB_0)\to\cB_0$ and the natural inclusion $\cB_0\mono\cB$. This provides a quasi-functor $\fH'\colon\cA\to\cB$ which induces yet another quasi-functor
\[
\xymatrix{
\fG\colon\hproj{\cA}\ar[rr]^-{\Ind(\fH')}&&\hproj{\cB}\ar[r]&\hproja{\cB},
}
\]
where the arrow on the right is just \eqref{eqn:alphaquot}. By passing to the homotopy categories, we finally have also the exact functor
\[
\fF:=\Ho(\fG)\colon\dgD(\cA)\longrightarrow\dgD_\alpha(\cB).
\]

\subsubsection*{The triangulated side of the story} Once the quasi-functor is constructed, there is some hard work to be done to solve two a priori non-trivial issues involving $\fF$.
\begin{itemize}
\medskip
\item[(1)] One first proves that $\fF$ factors through the quotient $\dgD(\cA)\to\dgD(\cA)/\cL$. This is done in \cite[Lemma 4.4]{CS6} and the proof heavily depends on the assumption (b.1). The output is another exact functor
\[
\fF'\colon\dgD(\cA)/\cL\longrightarrow\dgD_\alpha(\cB)
\]
which is again induced by a quasi-functor $\hproj{\cA}/\cL'\to\hproja{\cB}$. Here $\cL'$ is the full dg subcategory of $\hproj{\cA}$ corresponding to $\cL$ under the equivalence $\Ho(\hproj{\cA})\iso\dgD(\cA)$. 
\medskip
\item[(2)] Finally, one proves that $\fF'$ is an equivalence (see \cite[Proposition 4.5]{CS6}). The details cannot be provided here. One key feature that we would like to bring to the attention of the reader is that the objects in the image $\cS$ of the exact functor $\fF'\comp\fQ\comp\Yon$ do not form, in general, a set of $\alpha$-compact generators for $\dgD_\alpha(\cB)$ (see the discussion before \autoref{prop:inducedgen2}). But the objects of $\cS$ generate $\dgD_\alpha(\cB)$ (see \autoref{lem:inducedgen} and \autoref{rmk:gen}).
\end{itemize} 
\medskip
At this point, one has an invertible quasi-functor $\hproj{\cA}/\cL'\to\hproja{\cB}$ which, composed with the inverse of the invertible quasi-functor inducing \eqref{eqn:al}, yields an invertible quasi-functor $\hproj{\cA}/\cL'\to\cC$. Hence the dg enhancement $\cC$ is isomorphic in $\Hqe$ to the `standard' dg enhancement $\hproj{\cA}/\cL'$.

\subsection{The applications}\label{subsect:applicunbounded}

The last paragraph of \autoref{subsect:statemunbounded} shows that \autoref{thm:crit1} has the following consequence.

\begin{thm}[\cite{CS6}, Theorem A]\label{thm:main1}
If $\cG$ is a Grothendieck category, then $\D(\cG)$ has a unique enhancement.
\end{thm}

The geometric implications of this result are then easy to guess:
\begin{itemize}
\smallskip
\item If $X$ is an algebraic stack, the triangulated category $\D(\Qcoh(X))$ has a unique enhancement (see \cite[Corollary 5.4]{CS6});
\smallskip
\item If, in addition, $X$ is quasi-compact and with quasi-finite affine diagonal (this is true, in particular, if $X$ is a quasi-compact and semi-separated scheme), then the full triangulated subcategory $\Dq(X)$ of $\D(\Mod{\ko_X})$ consisting of complexes with quasi-coherent cohomology has a unique enhancement (see, again, \cite[Corollary 5.4]{CS6});
\smallskip
\item If $X$ is scheme and $\alpha$ is an element in the Brauer group $\Br(X)$ of $X$, then the twisted derived category $\D(\Qcoh(X,\alpha))$ has a unique enhancement (see \cite[Corollary 5.7]{CS6}).
\end{itemize}
\smallskip
Indeed, under the above assumptions, $\Qcoh(X)$ and $\Qcoh(X,\alpha)$ are Grothendieck categories (see, for example, \cite[Tag 06WU]{SP}).

The non-expert reader can have a look at \cite{LMB} for the general theory of stacks and at \cite{Cal} for the theory of schemes endowed with a twist from their Brauer groups. Nevertheless, recall that if $X$ is a scheme and 
$\alpha$ is an element in $H^2_{\text {\'et}}(X,\ko_X^*)$ (i.e.\ an
element in the Brauer group $\Br(X)$ of $X$), we may see $\alpha$ as a \v{C}ech
2-cocycle $\{\alpha_{ijk}\in\Gamma(U_i\cap U_j\cap
U_k,\ko^*_X)\}$ with $X=\bigcup_{i\in I} U_i$  an
appropriate  open cover in the \'etale topology. An \emph{$\alpha$-twisted
quasi-coherent sheaf} $E$ consists of pairs $(\{E_i\}_{i\in
I},\{\varphi_{ij}\}_{i,j\in I})$ such that the $E_i$ are
quasi-coherent sheaves on $U_i$ and $\varphi_{ij}:{E}_j|_{U_i\cap
U_j}\rightarrow{E}_i|_{U_i\cap U_j}$ are isomorphisms satisfying
the following conditions:
\begin{itemize}
\item $\varphi_{ii}=\id$;
\item $\varphi_{ji}=\varphi_{ij}^{-1}$;
\item $\varphi_{ij}\comp\varphi_{jk}\comp\varphi_{ki}=\alpha_{ijk}\cdot\id$.
\end{itemize}
The abelian category of such
$\alpha$-twisted quasi-coherent sheaves on $X$ is denoted by $\Qcoh(X,\alpha)$. If the sheaves $E_i$ in the above definition are coherent rather than just quasi-coherent, we get the notion of an $\alpha$-twisted coherent sheaf and the corresponding abelian category $\Coh(X,\alpha)$.

\subsection{Open questions}\label{subsect:openunbounded}

There is another interesting triangulated category which is associated to a Grothendieck category $\cG$. Let $\Inj(\cG)$ be the full subcategory of $\cG$ consisting of injective objects. By \autoref{algchar}, the triangulated category $\Kc(\Inj(\cG))$ is algebraic.

\begin{qn}\label{qn:Kinj}
Does $\Kc(\Inj(\cG))$ have a unique enhancement?
\end{qn}

To motivate our interest in $\Kc(\Inj(\cG))$, observe that such a category played an import role in \cite{BKI}, which deals with modular representations of finite groups. On the other hand, $\Kc(\Inj(\cG))$ has been used to reformulate Grothendieck duality in \cite{NInj}.

\smallskip

It is rather clear that major improvements could come by removing or rather weakening (b.1) in \autoref{thm:crit1}. Indeed, it is very likely that if this can be done, then one could hope to get interesting results about uniqueness of enhancements of some localizing subcategories of categories with unique enhancements. Thus we would get interesting 
answers to \autoref{qn:uniqloc}.

\smallskip

There are other interesting full subcategories of $\D(\Qcoh(X))$, for $X$ an algebraic stack, which are not included in the previous discussion. The first one is $\Db(\Qcoh(X))$, which is clearly not closed under small coproducts. Moreover, by \eqref{eqn:alpcom} and for any sufficiently large regular cardinal one can take $\D(\Qcoh(X)^\alpha)\iso\D(\Qcoh(X))^\alpha$. Recall that $\Qcoh(X)^\alpha$ is an abelian category for this choice of $\alpha$.

Clearly the problem of the uniqueness of the enhancement of these categories cannot be treated using \autoref{thm:crit1}. Thus it makes sense to ask the following.

\begin{qn}\label{qn:dbqcohalpha}
If $X$ is an algebraic stack, do $\Db(\Qcoh(X))$ and $\D(\Qcoh(X))^\alpha$ have unique enhancements, for $\alpha$ a sufficiently large regular cardinal?
\end{qn}

One key property that makes $\Db(\Qcoh(X))$ interesting for our purposes is that there is a natural exact equivalence $\Db(\Qcoh(X))^c\iso\Db(X)$, if $X$ is a noetherian separated scheme (see \cite[Corollary 6.17]{R}). Notice that, in this setting, $\Db(\Qcoh(X))^c$ is the full subcategory of $\Db(\Qcoh(X))$ consisting of those objects $C$ such that $\Db(\Qcoh(X))(C,\farg)$ commutes with small coproducts existing in $\Db(\Qcoh(X))$.

\autoref{qn:Kinj} and \autoref{qn:dbqcohalpha} will be addressed in a separate paper.

\section{Uniqueness of enhancements for compact objects}\label{sect:uniqcatcomp}

Let us now consider the full subcategory of compact objects in the derived category of a Grothendieck category. One of the main aims of this section is to try to clarify why this new situation is a bit more intricate and illustrate some open problems.

\subsection{The statements}\label{subsect:statemcomp}

The analogue of \autoref{thm:crit1} in the context of this section is the following.

\begin{thm}[\cite{LO}, Theorem 2]\label{thm:LOcomp}
Let $\cA$ be a small category which we consider as a dg category concentrated in degree zero and let $\cL$ be a localizing subcategory of $\dgD(\cA)$ such that:
\begin{itemize}
\item[{\rm (a.3)}] $\cL^c=\cL\cap\dgD(\cA)^c$ and $\cL^c$ satisfies {\rm (G1)} in $\cL$;
\item[{\rm (b.3)}] $\dgD(\cA)/\cL(\fQ(\Yon(A)),\fQ(\Yon(A'))[k])=0$, for all $A,A'\in\cA$ and all integers $k<0$.
\end{itemize}
Then $(\dgD(\cA)/\cL)^c$ has a unique enhancement.
\end{thm}

\begin{remark}
Notice that (b.3) is exactly the same as (b.2) in \autoref{thm:LO1}. Moreover, (a.3) could be replaced by the following weaker assumption: $\cL\cap\dgD(\cA)^c$ satisfies {\rm (G1)} in $\cL$. But the argument in the proof of \cite[Corollary 6.7]{CS6} shows that, if this is true, then automatically $\cL^c=\cL\cap\dgD(\cA)^c$.
\end{remark}

Suppose now that we want to deal with the enhancements of $\D(\cG)^c$, where $\cG$ is a Grothendieck category. We already know that we can pick a small set of generators $\cA$ of $\cG$ such that one gets an exact equivalence $\D(\cG)^c\iso(\dgD(\cA)/\cL)^c$, for some localizing subcategory $\cL$ of $\dgD(\cA)$.

The discussion at the end of \autoref{subsect:proofunbounded} should have clarified that (b.3) is automatically satisfied in this situation. Unfortunately, (a.3) is not easily verified and this is the reason why $\cA$ has to be chosen carefully. In conclusion, the following is the result that can be deduced from \autoref{thm:LOcomp}.

\begin{thm}[\cite{CS6}, Theorem B]\label{thm:crit2}
Let $\cG$ be a Grothendieck category with a small set $\cA$ of generators such that
\begin{enumerate}
\item $\cA$ is closed under finite coproducts;
\item Every object of $\cA$ is a noetherian object in $\cG$;
\item If $f\colon A'\epi A$ is an epimorphism of $\cG$ with
$A,A'\in\cA$, then $\ker f\in\cA$;
\item For every $A\in\cA$ there exists $N(A)>0$ such that
$\D(\cG)\left(A,\sh[N(A)]{A'}\right)=0$ for every $A'\in\cA$.
\end{enumerate}
Then $\D(\cG)^c$ has a unique enhancement.
\end{thm}

Showing that $\cL$ satisfies (a.3), under the assumptions (1)--(4), requires a big amount of quite technical work that cannot be summarized here. Nonetheless, in \autoref{subsect:conjcomp} we will explain why some assumptions on $\cA$ are very plausible.

\begin{remark}\label{rmk:semis}
We observed in \cite[Remark 6.8]{CS6} that \autoref{thm:crit2} can be refined a bit and conclude that, under the same assumptions, $\D(\cG)^c$ has as a semi-strongly unique enhancement. For this, one use \cite[Theorem 6.4]{LO} instead of \autoref{thm:LOcomp}.
\end{remark}

\begin{remark}\label{rmk:cpt}
It is easy to deduce directly from \autoref{thm:crit2} that $\D(\cG)$ has a unique enhancement, if $\cG$ is a Grothendieck category satisfying the assumptions (1)--(4). Indeed, assume that $(\cC_1,\fE_1)$ and $(\cC_2,\fE_2)$ are enhancements of $\D(\cG)$ and define $\cD_i$ to be the full dg subcategory of $\cC_i$ consisting of the objects in $\Ho(\cC_i)^c$. Set $\fF_i:=\fE_i\rest{\cD_i}$. It is clear that $(\cD_1,\fF_1)$ and $(\cD_2,\fF_2)$ are enhancements of $\D(\cG)^c$. Since, by \autoref{thm:crit2}, the latter category has a unique enhancement, there exists an isomorphism $\fF\in\Hqe(\cD_1,\cD_2)$. By the discussion in \autoref{sect:dgenhancements}, we get an isomorphism $\fF':=\Ind(\fF)\in\Hqe(\hproj{\cD_1},\hproj{\cD_2})$. 

Finally, we observed above that (1)--(4) in \autoref{thm:crit2} imply (a.3) in \autoref{thm:LOcomp}. By \autoref{thm:compquot}, the triangulated category $\D(\cG)$ is generated by the objects in $\D(\cG)^c$. Hence, by \cite[Proposition 1.17]{LO}, we get an isomorphism $\fG_i\in\Hqe(\hproj{\cD_i},\cC_i)$, for $i=1,2$. Thus $\fG_2\comp\fF'\comp\fG_1^{-1}$ is the isomorphism in $\Hqe(\cC_1,\cC_2)$ we were looking for.

Unfortunately we do not know any strategy to deduce from \autoref{thm:main1} the uniqueness of the enhancement of $\D(\cG)^c$.
\end{remark}

\smallskip

Back to the geometric situation, recall that, given a commutative ring $R$, a complex $P\in\D(\Mod{R})$ is \emph{perfect} if it is quasi-isomorphic to a bounded complex of finitely generated projective $R$-modules. According to the terminology in \cite{HR}, a complex $P\in\Dq(X)$, where $X$ is an algebraic stack, is \emph{perfect} if for any smooth morphism $\spec(R)\to X$, where $R$ is a commutative ring, the complex of $R$-modules $\rd\Gamma(\spec(R),P\rest{\spec(R)})$ is perfect. We set $\Dp(X)$ to be the full subcategory of $\Dq(X)$ consisting of perfect complexes.

Now, a quasi-compact and quasi-separated algebraic stack $X$ is \emph{concentrated} if $\Dp(X)\subseteq\Dq(X)^c$ (notice that this is always the case if $X$ is a scheme). On the other hand, by \cite[Lemma 4.4]{HR}, the other inclusion $\Dq(X)^c\subseteq\Dp(X)$ holds as well. Moreover, if $X$ has also quasi-finite affine diagonal, then the natural functor $\D(\Qcoh(X))\to\Dq(X)$ is an exact equivalence (this follows from \cite[Theorem A]{HR} and \cite[Theorem 1.2]{HNR}). Therefore, if $X$ is a concentrated algebraic stack with quasi-finite affine diagonal, then there is a natural exact equivalence
\begin{equation*}\label{perfcpt}
\Dp(X)\iso\D(\Qcoh(X))^c.
\end{equation*}
Finally, if an algebraic stack $X$ has the property that $\Qcoh(X)$ is generated, as a Grothendieck category, by a small set of objects contained in $\Coh(X)\cap\Dp(X)$, we say that \emph{$X$ has enough perfect coherent sheaves}.

We observed in \autoref{subsect:applicunbounded} that $\Qcoh(X)$ is a Grothendieck category.
Hence, it is easy to guess that the main geometric application of the general results above is:

\begin{cor}[\cite{CS6}, Proposition 6.10]\label{cor:geocomp}
If $X$ is a noetherian concentrated algebraic stack with quasi-finite affine diagonal and with enough perfect coherent sheaves, then $\Dp(X)$ has a semi-strongly unique enhancement.
\end{cor}

Here we are using \autoref{rmk:semis} to strengthen the result \cite[Proposition 6.10]{CS6} and prove semi-strongly uniqueness rather than uniqueness. To give an idea of how the proof goes, we just mention that we can take $\cA$ to be the full subcategory of $\Qcoh(X)$ whose set of objects is obtained by taking a representative in each isomorphism class of objects in $\Coh(X)\cap\Dp(X)$. Since $X$ has enough perfect coherent sheaves, $\cA$ is a set of generators of $\Qcoh(X)$, as a Grothendieck category.

The reader which feels a bit uneasy with the language of stacks can be reassured by the simpler case where $X$ is a noetherian scheme with enough locally free sheaves. In this case, $\Dp(X)$ has a unique enhancement (see \cite[Corollary 6.11]{CS6}). A scheme \emph{has enough locally free sheaves} if, for any finitely presented sheaf $F$, there is an epimorphism $E\epi F$ in $\Qcoh(X)$, where $E$ is locally free of finite type.

\begin{remark}\label{rmk:Tot}
It should be clarified that no example of a quasi-compact and semi-separated scheme without enough locally free sheaves is known. Actually, in \cite{Tot} the existence of enough locally free sheaves under these assumptions is left as an open question.
\end{remark}

\begin{remark}\label{rmk:coh}
The same circle of ideas applies to $\Db(X)$, the bounded derived category of coherent sheaves on a scheme $X$. Indeed, if $X$ is a noetherian scheme with enough locally free sheaves, then $\Db(X)$ has a unique enhancement (see \cite[Corollary 7.2]{CS6}).

For this, we use the notion of \emph{compactly approximated} object in $\D(\Qcoh(X))$ rather than the one of compact object. We do  not define it here.
\end{remark}

\subsection{The importance of being compact}\label{subsect:conjcomp}

Let us go back to the general setting where $\cT$ is a compactly generated triangulated category with small coproducts.

\begin{definition}\label{def:smashing}
A localizing subcategory $\cL$ of $\cT$ is a \emph{smashing subcategory} if the inclusion functor $\iota\colon\cL\mono\cT$ has a right adjoint $\iota^R$ which commutes with small coproducts.
\end{definition}

\begin{lem}\label{lem:smash}
Let $\cT$ be a compactly generated triangulated category with small coproducts and let $\cL$ be a localizing subcategory of $\cT$. If $\cL^c=\cL\cap\cT^c$ and $\cL^c$ satisfies {\rm (G1)} in $\cL$, then $\cL$ is a smashing subcategory of $\cT$.
\end{lem}

\begin{proof}
First notice that, $\cL$ being compactly generated, $\iota^R$ exists by Brown representability (see, for example, \cite[Section 8.2]{N2}).
  
Given $\{X_i:i\in I\}$ a small set of objects in $\cT$ and $L\in\cL^c$, there is a sequence of natural isomorphisms
\[
\cL\biggl(L,\iota^R\Bigl(\Plus_{i\in I}X_i\Bigr)\biggr)\iso\cT\Bigl(\iota(L),\Plus_{i\in I}X_i\Bigr)\iso\Plus_{i\in I}\cT\bigl(\iota(L),X_i\bigr)\iso\Plus_{i\in I}\cL\bigl(L,\iota^R(X_i)\bigr)\iso\cL\Bigl(L,\Plus_{i\in I}\iota^R(X_i)\Bigr).
\]
Here the second isomorphism uses that $\cL^c=\cL\cap\cT^c$.

Since $\cL^c$ satisfies (G1) in $\cL$, we deduce from \autoref{rmk:gen} that
\[
\cL\left( L,\iota^R\left(\Plus_{i\in I}X_i\right)\right)\iso\cL\left( L,\Plus_{i\in I}\iota^R\left(X_i\right)\right), 
\]
for all $L\in\cL$. Hence $\iota^R\left(\Plus_{i\in I}X_i\right)\iso\Plus_{i\in I}\iota^R\left(X_i\right)$ by Yoneda's lemma.
\end{proof}

It follows that $\cL$ in \autoref{thm:LOcomp} is a smashing subcategory. The converse of \autoref{lem:smash} is the content of the following.

\begin{conj}[Telescope Conjecture]\label{conj:telescope}
If $\cT$ is a compactly generated triangulated category with small coproducts and $\cL$ is a smashing subcategory of $\cT$, then $\cL^c=\cL\cap\cT^c$ and $\cL^c$ satisfies {\rm (G1)} in $\cL$.
\end{conj}

\begin{remark}\label{rmk:telescope}
It should be noted that the original formulation of the above conjecture would be: If $\cT$ is a compactly generated triangulated category with small coproducts and $\cL$ is a smashing subcategory of $\cT$, then $\cL$ is generated by a small set of objects in $\cL\cap\cT^c$. The conjecture is presented in this form in \cite{Ksm2} (and it is based on conjectures by Ravenel \cite[1.33]{Ra}).

Now it is clear that it is equivalent to take the whole $\cL\cap\cT^c$ as a set of generators. On the other hand, if the objects in $\cL\cap\cT^c$ generate $\cL$, then they satisfy (G1) in $\cL$ as well. We want to prove that if $\cL\cap\cT^c$ satisfies (G1) in $\cL$, then it generates $\cL$ as well. Indeed, if $\cN$ is the smallest localizing subcategory of $\cT$ containing $\cL\cap\cT^c$, then, by \autoref{thm:compquot}, it is well generated. Hence Brown representability holds for $\cN$. At this point, one copies verbatim the same argument as in the proof of \cite[Proposition 5.1]{P}.

In conclusion, if $\cL\cap\cT^c$ satisfies (G1) in $\cL$, then it is easy to show that $\cL^c=\cL\cap\cT^c$ (see the proof of \cite[Corollary 6.7]{CS6}). In other words, the classical version of the Telescope Conjecture is equivalent to \autoref{conj:telescope}.
\end{remark}

Long ago, Keller produced a simple counterexample to this conjecture \cite{KTel}. Hence we cannot expect that verifying (a.3) can be an easy task, in general. In a sense, this motivates the fact that the presence of the assumptions (1)--(4) in \autoref{thm:crit2} should not be surprising.

Anyway, we do not feel that (1)--(4) are the sharpest assumptions to make. In particular, we believe that the following is a rather natural problem to consider.

\begin{qn}\label{qn:smash}
Is a result like \autoref{thm:LOcomp} true if we change (a.3) to the assumption that $\cL$ is a smashing subcategory?
\end{qn}

If this were true, then in \autoref{cor:geocomp} one could replace `noetherian' with `quasi-compact, semi-separated'. Indeed, with the above choice of $\cA$, the quotient functor $\fQ\colon\dgD(\cA)\to\dgD(\cA)/\cL$ sends $\dgD(\cA)^c$ to $(\dgD(\cA)/\cL)^c$ (this follows quite easily from \cite[Corollary 5.3]{CS6}). It is an exercise with the definition of compact objects to prove that, if this is true, then $\cL$ is smashing.

An approach to \autoref{qn:smash} could be via \cite{Ksm1} and \cite{Ksm2}.

\subsection{Strong uniqueness}\label{subsect:strongly}

We finish this section recalling that a result similar to \autoref{cor:geocomp} can be proven for the strong uniqueness of enhancements.

\begin{thm}[\cite{LO}, Theorem 2.14]\label{thm:LOstrong}
Let $X$ be a projective scheme over a field $\K$ such that the maximal $0$-dimensional torsion subsheaf $T_0(\ko_X)$ of $\ko_X$ is trivial. Then $\Dp(X)$ and $\Db(X)$ have strongly unique enhancements.
\end{thm}

In a sense, \autoref{thm:LOstrong} can be deduced from \autoref{cor:geocomp} (and \autoref{rmk:coh}) using two main additional ingredients: ample sequences and convolutions. Indeed, assume that $X$ is as in \autoref{thm:LOstrong} and that we are given a pretriangulated dg category $\cC$ and an exact equivalence
\[
\fF\colon\Dp(X)\to\Ho(\cC).
\]
\autoref{cor:geocomp} provides an isomorphism $f\in\Hqe(\Perf(X),\cC)$, for any enhancement $\Perf(X)$ of $\Dp(X)$. How do we show that, roughly speaking, there is an isomorphism of exact functors $\fF\iso\Ho(f)$? As we mentioned, this is achieved by a sort of standard procedure initiated in \cite{Or} which uses ample sequences and convolutions.

These two techniques are reviewed in some detail in Sections 5.2.2 and 5.2.3 of \cite{CS3}. The expert reader can also guess from \cite{CS3} how they are used to carry out the proof of \autoref{thm:LOstrong}. Here we just recall that, given an abelian category
$\cA$ with finite dimensional Hom-spaces, a subset
$\{P_i\}_{i\in\ZZ}\subset\cA$ is an \emph{ample sequence}
if, for any $B\in\cA$, there exists an integer $N(B)$ such
that, for any $i\leq N(B)$,
\begin{enumerate}
\item\label{ample1} The natural morphism $\cA(P_i,B)
\otimes P_i\to B$ is surjective;

\item\label{ample2} If $j\ne0$ then
$\Db(\cA)(P_i,B[j])=0$;

\item\label{ample3} $\cA(B,P_i)=0$.
\end{enumerate}

If $X$ is a smooth projective scheme of positive dimension with an ample line bundle $L$, then $\{L^{\otimes n}:n\in\ZZ\}$ is an ample sequence in $\Coh(X)$. If $X$ is singular, this is still true if we make the further technical assumption that the maximal $0$-dimensional torsion subsheaf $T_0(\ko_X)$ of $\ko_X$ is trivial.

\begin{remark}\label{rmk:twist}
(i) There is another situation where strong uniqueness can be proven. Indeed, assume that $X$ is a smooth projective scheme over a field $\K$ and $\alpha$ is an element in the Brauer group $\Br(X)$ of $X$. Then, by \cite[Lemma 2.3]{CS5}, the abelian category $\Coh(X,\alpha)$ has an ample sequence. Then the same argument as in the proof of \autoref{thm:LOstrong} shows that $\Db(X,\alpha)$ has a strongly unique enhancement.

(ii) In \cite[Theorem 1.2]{CS}, we considered the category $\mathbf{Perf}_{Z}(X)$ consisting of perfect complexes on a quasi-projective scheme $X$ with topological support on a projective subscheme $Z$. More specifically, we proved that if $\ko_{iZ}\in\Dp(X)$, for all $i>0$ and $T_0(\ko_Z)= 0$, then $\mathbf{Perf}_Z(X)$ has a strongly unique enhancement. We let the interested reader consult \cite{CS3,CS} for the correct definition of the triangulated category $\mathbf{Perf}_{Z}(X)$.
\end{remark}

\section{Bad news again: lifting functors}\label{sect:functors}

In this section we give a quick update of \cite{CS3}. In particular, we explain some recent results showing that the existence (and uniqueness) of dg lifts of exact functors in geometric contexts is not always available.

\subsection{The triangulated case}\label{subsect:funtria}

Following \cite{CS3}, let us remind the reader of a very simple way to summarize the complete picture concerning Fourier--Mukai functors between `geometric' categories.

To be more precise, if $X_1$ and $X_2$ are smooth projective schemes over a field $\K$, we consider the exact functor
\begin{equation*}\label{eqn:FM}
\FM{E}(-):=\rd(p_2)_*(E\lotimes p_1^*(-))\colon\Db(X_1)\to\Db(X_2),
\end{equation*}
where $p_i\colon X_1\times X_2\to X_i$ is the natural projection and $E\in\Db(X_1\times X_2)$.

\begin{definition}\label{def:FM}
An exact functor $\fun{F}\colon\Db(X_1)\to\Db(X_2)$ is a \emph{Fourier--Mukai functor} (or of \emph{Fourier--Mukai type}) if there exists $E\in\Db(X_1\times X_2)$ and an isomorphism of exact functors $\fun{F}\iso\FM{E}$. The object $E$ is called \emph{Fourier--Mukai kernel}.
\end{definition}

\begin{remark}\label{rmk}
The relevance of these functors has been explained in \cite{CS3} (see also \cite{H} for a systematic analysis of their main geometric applications).

Moreover, it was conjectured by Kawamata \cite[Conjecture 1.5]{Ka1} that, given a smooth projective variety $X$ over a field, there are, up to isomorphism, only a finite number of smooth projective varieties $Y$ with an exact equivalence $\Db(X)\iso\Db(Y)$. Later, in \cite{AT}, it was shown that, again up to isomorphism, the smooth projective varieties $Y$ as above are at most countably many. In \cite{L}, the author shows a countable family of non-isomorphic threefolds with equivalent derived categories. This falsifies Kawamata's conjecture and updates our discussion in \cite[Section 2.2]{CS3}.
\end{remark}

Here we stick to some more foundational questions related to the subject of this survey.
Indeed, it makes sense to consider the category
$\ExFun(\Db(X_1),\Db(X_2))$ of exact functors
between $\Db(X_1)$ and $\Db(X_2)$ (with morphisms the natural
transformations compatible with shifts) and to define the functor
\begin{equation}\label{eqn:fun}
\FM[X_1\to X_2]{\farg}\colon\Db(X_1\times X_2)\lto\ExFun(\Db(X_1),\Db(X_2)).
\end{equation}
It sends $E\in\Db(X_1\times X_2)$ to the Fourier--Mukai functor $\FM{E}$.

The behavior of this functor is as bad as possible. To be more precise, in \cite{CS1}  we proved the following:
\begin{itemize}
\medskip
\item[(a)] $\FM[X_1\to X_2]{\farg}$ is not essentially injective (\cite[Theorem 1.1]{CS1}). In particular, for any elliptic curve $E$ over an algebraically closed
field there are non-isomorphic objects $E_1,E_2\in\Db(E\times E)$ such that $\FM{E_1}\iso\FM{E_2}$.
\item[(b)] $\FM[X_1\to X_2]{\farg}$ is neither full nor faithful (\cite[Proposition 2.3]{CS1}).
\medskip
\item[(c)] $\ExFun(\Db(X_1),\Db(X_2))$ does not have a triangulated structure making $\FM[X_1\to X_2]{\farg}$ exact (\cite[Corollary 2.7]{CS1}).
\end{itemize}

\begin{remark}\label{rmk:uniqcoh}
A partial repair to (a) is provided in \cite[Theorem 1.2]{CS1} where it is proved that the cohomology sheaves of a Fourier--Mukai functor are uniquely determined, up to isomorphism. More precisely, let $X_1$ and $X_2$ be projective schemes defined over a field and let $E_1,E_2\in\Db(X_1\times X_2)$ be such that they define exact functors $\FM{E_i}\colon\Dp(X_1)\to\Db(X_2)$, for $i=1,2$. If there is an isomorphism of exact functors $\FM{E_1}\iso\FM{E_2}$, then $H^j(E_1)\iso H^j(E_2)$, for all $j\in\ZZ$.
\end{remark}

\subsection{The counterexamples: non-liftable exact functors}\label{subsect:funcounter}

The most interesting question concerns the essential surjectivity of $\FM[X_1\to X_2]{\farg}$. This has been investigated during the last years and it has been recently clarified that the answer has to be negative. We discuss here the only two known examples of exact functors between the derived categories of smooth projective schemes which are not of Fourier--Mukai type.

\subsubsection*{Quadrics in $\PP^4$ (Rizzardo and Van den Bergh, \cite{RvdB})}\label{ssubsection:RvdB}

If $\K$ is an algebraically closed field of characteristic $0$ and $X$ is a smooth quadric in $Y=\PP^4$, then Rizzardo and Van den Bergh proved that there exists an exact functor $\Db(X)\to\Db(Y)$ which is not of Fourier--Mukai type. Unfortunately, not only the proof of this result, but even the definition of the functor is rather involved, as it uses sophisticated techniques from deformation theory. So we will limit ourselves to give a rough idea of their construction, inviting the interested reader to consult directly \cite{RvdB} for more details.

Consider more generally a closed immersion $f\colon X\to Y$ between smooth projective varieties and choose a finite affine open cover $\{V_1,\dots,V_r\}$ of $Y$. Denoting by $\I$ the set of non-empty subsets of $\{1,\dots,r\}$, we define as usual $V_I:=\cap_{i\in I}V_i$ for every $I\in\I$. To this cover one can associate a ($\K$-linear) category $\Y$ with the same objects as $\I$ and such that $\Y(I,J)$ (for $I,J\in\I$) is $\so_Y(V_J)$ if $I\subseteq J$ and $0$ otherwise. It is easy to see that there is a natural exact and fully faithful functor $\epsilon^*\colon\Qcoh(Y)\to\Mod{\Y\opp}$, which sends $E\in\Qcoh(Y)$ to the functor which associates to any $I\in\I$ the sections $E(V_I)$. This induces an exact equivalence between $\D(\Qcoh(Y))$ and $\Dq(\Y)$. Here $\Dq(\Y)$ is the full subcategory of $\D(\Mod{\Y\opp})$ with objects the complexes whose cohomologies are in the essential image of $\epsilon^*$. Similarly, to the induced (finite affine) cover $\{f^{-1}(V_1),\dots,f^{-1}(V_r)\}$ of $X$ one can associate a category $\X$, and $f$ naturally yields a ($\K$-linear) functor, again denoted by $f\colon\Y\to\X$, by abuse of notation. Indeed, $f$ is the identity on objects, whereas on morphisms it is given by the structure maps $f^{\#}(V_J)\colon\so_Y(V_J)\to(f_*\so_X)(V_J)\iso\so_X(f^{-1}(V_J))$.

Now, given a line bundle $M$ on $X$ and $\eta\in\HH^n(X,M):=\Ext^n_{X\times X}(\diag_*\so_X,\diag_*M)$, they introduce an $A_{\infty}$-deformation of $\X$, given by an $A_{\infty}$-category $\X_{\eta}$ together with an $A_{\infty}$-functor $\X_{\eta}\to\X$. It has the property that there exists an $A_{\infty}$-functor $\tilde{f}\colon\Y\to\X_{\eta}$ such that the diagram
\[
\xymatrix{
\X_{\eta} \ar[dr] & & \Y \ar[ll]_-{\tilde{f}} \ar[dl]^-{f}\\
 & \X
}
\]
commutes, up to isomorphism, if $f_*(\eta)=0\in\HH^n(Y,f_*(M))$. Moreover, under the additional assumption $n\ge\dim(X)+3$, there is a suitable exact functor $\fL\colon\Db(\Qcoh(X))\to\Dq(\X_{\eta})$ (where the latter category can be defined in analogy with $\Dq(\Y)$ above). Actually the proof of the existence of $\fL$, which is obtained as an extension of a natural functor defined on $\Inj(\Qcoh(X))$, is the more technical part of the paper. Finally, they show that the composition
\[
\Db(\Qcoh(X))\mor{\fL}\Dq(\X_{\eta})\mor{\tilde{f}_*}\Dq(\Y)\iso\D(\Qcoh(Y))
\]
restricts to the desired exact functor not of Fourier--Mukai type $\Db(X)\to\Db(Y)$ when $X$ is a smooth quadric in $Y=\PP^4$, $M=\omega_X^{\otimes2}$ and $0\ne\eta\in\HH^6(X,\omega_X^{\otimes2})\iso\K$ (in which case the condition $f_*(\eta)=0$ is satisfied).

\subsubsection*{Flag varieties (Vologodsky, \cite{V})}\label{ssubsection:V}

Quite recently, Vologodsky proposed a different (and, in a sense, simpler) example of an exact functor which is not of Fourier--Mukai type. His example is completely of geometric nature and very natural. Thus, it puts a bit in the shade the general (too) optimistic belief that even though $\FM[X_1\to X_2]{\farg}$ is not essentially surjective, algebraic geometers are safe: all exact functors appearing in their life are of Fourier--Mukai type.

For an integer $n>2$ and a prime $p$, consider the general linear group $\GL_n(\ZZ_p)$ over the $p$-adic numbers and take $B$ to be a Borel subgroup of $\GL_n(\ZZ_p)$. The quotient $\Fl_n:=\GL_n(\ZZ_p)/B$ is a flag variety over $\ZZ_p$. Set $Y:=\Fl_n\times_{\spec(\ZZ_p)} \Fl_n$.

Now, let $\K=\mathbb{F}_p$ and let $X:=Y\times_{\spec(\ZZ_p)}\spec(\K)$. In this situation, we have a natural closed embedding $\iota\colon X\mono Y$
and a $\K$-linear exact functor $\fG\colon\Db(X)\to\Db(X)$ defined as
\[
\fG:=\ld\iota^*\comp\iota_*.
\]
The claim is that $\fG$ is not of Fourier--Mukai type if regarded as a $\K$-linear functor. Notice that, on the other hand, it is of Fourier--Mukai type as a $\ZZ_p$-linear functor.

\subsubsection*{Life after the counterexamples}\label{ssubsection:after}

In view of the negative answer to the possibility of describing all exact functors between $\Db(X_1)$ and $\Db(X_2)$ as Fourier--Mukai functors, it becomes an interesting and challenging problem to characterize all exact functors which are of Fourier--Mukai type.

There has been a lot of work in this direction in recent years showing that interesting classes of exact functors are of this form. More precisely, a (non-exhaustive) list of results is as follows (a survey of these results is already contained in \cite{CS3}):
\begin{itemize}
\smallskip
\item[(1)] {\it Smooth projective schemes:} In \cite{Or}, it was proved that all fully faithful exact functors between $\Db(X_1)$ and $\Db(X_2)$, with $X_1$ and $X_2$ smooth projective schemes over a field $\K$, are of Fourier--Mukai type with unique (up to isomorphism) Fourier--Mukai kernel. The result has been generalized to the twisted setting under milder assumptions on the functor in \cite{CS5}. A generalization of Orlov's result to smooth stacks, which are obtained as global quotients, is proved in \cite{Ka}.
\smallskip
\item[(2)] {\it The singular case:} In \cite{LO}, the case of projective (non-necessarily smooth) schemes is treated. Among other things, the authors show that all fully faithful exact functors between $\Dp(X_1)$ and $\Dp(X_2)$, with $X_i$ projective over a field and such that the maximal $0$-dimensional torsion subsheaf $T_0(\ko_{X_1})$ of $\ko_{X_1}$ is trivial, are of Fourier--Mukai type. The uniqueness of the Fourier--Mukai kernel is proved in \cite[Remark 5.7]{CS}. One of the major contributions of \cite{LO} consists in showing that a fruitful approach involves the use of dg lifts.
\smallskip
\item[(3)] {\it The supported case:} In \cite{CS}, the techniques introduced in \cite{LO} are enforced and extended to deal with exact functors (with some special assumptions) between the categories of perfect complexes with cohomologies supported on projective schemes. 
\end{itemize}

\begin{remark}\label{rmk:students}
(i) All the results in (1)--(3) use in a key way the notion of ample sequence (see \autoref{subsect:strongly} for a quick discussion about this). Inevitably, this implies that, using these techniques, one can hope to work only with projective schemes. A different approach may consist in using \emph{indecomposable objects} rather than ample sequences. An object $C$ in a triangulated category $\cT$ is indecomposable if it is not isomorphic to a direct sum $C_1\oplus C_2$ with $C_1$ and $C_2$ non-trivial. This was successfully pursued in \cite{AM} for the special case of fully faithful functors $\fF\colon\Dp(X)\to\Db(Y)$, where $X=\spec\left(\K[\ep]\right)$ (as in \autoref{nonuniquot}) and $Y$ is a noetherian separated scheme defined over a field $\K$. There is a chance that this viewpoint can be fruitful in other geometric situations involving schemes which are not projective. 

(ii) Understanding the uniqueness of Fourier--Mukai kernels is an interesting problem in itself. In \cite{Ge} it is suggested that one can study it by lifting exact functors to $A_\infty$-functors.
\end{remark}

\subsection{The dg case}\label{subsect:fundg}

The bad picture described in the previous sections should be compared to the idyllic situation we encounter when we move to the dg context.

In principle, there are two different contexts where one can try to define an appropriate notion of dg Fourier--Mukai functor. The first one consists in taking the category $\dgCat$ of small dg categories which are $\K$-linear, for a commutative ring $\K$. Having in mind the idea of comparing the dg and triangulated worlds it is rather clear that this is not the right perspective to take. Indeed, in this process, we need to invert quasi-equivalences and thus we are forced to work with the localization $\Hqe$ rather than with $\dgCat$.

The second context, where all these problems are overcome, is $\Hqe$. This is what we are going to investigate in the rest of this section. The main result in this sense is the following one which is due to To\"en.

\begin{thm}[\cite{To}]\label{thm:Toen}
Let $\cC_1$, $\cC_2$ and $\cC_3$ be three dg categories over a commutative ring $\K$. Then there exists a natural bijection
\begin{equation}\label{eqn:firstpart}
\xymatrix{
\Hqe(\cC_1,\cC_2)\ar@{<->}[rr]^-{1:1}&&\Iso(\Ho(\rqr{\cC_1\opp\lotimes\cC_2})).
}
\end{equation}
Moreover, the dg category $\IHom(\cC_3,\cC_2):=\rqr{\cC_3\opp\lotimes\cC_2}$ yields a natural bijection
\begin{equation*}\label{eqn:ultima}
\xymatrix{
\Hqe(\cC_1\lotimes\cC_3,\cC_2)\ar@{<->}[rr]^-{1:1}&&\Hqe(\cC_1,\IHom(\cC_3,\cC_2))
}
\end{equation*}
proving that the symmetric monoidal structure on $\Hqe$ is closed.
\end{thm}

The result has been reproved in a rather elementary way in \cite{CS4}. Here we take this perspective to illustrate how the bijection \eqref{eqn:firstpart} is defined. To see this, consider first the bijection
\begin{equation*}\label{eqn:intHoms1}
	\xymatrix{
	\Hqe(\cC_1,\cC_2)\ar@{<->}[r]^-{1:1}&
	\Hqe(\cC_1,\essim{\cC_2}).
	}
\end{equation*}
The inclusion $\essim{\cC_2}\mono\hproj{\cC_2}$ induces a natural injection
\begin{equation}\label{eqn:intHoms2}
	\xymatrix{	
	\Hqe(\cC_1,\essim{\cC_2})\ar@{^{(}->}[r]&
	\Hqe(\cC_1,\hproj{\cC_2}).
	}
\end{equation}
To conclude, observe that there is a natural bijection
\begin{equation}\label{eqn:intHoms3}
	\xymatrix{
\Iso(\Ho(\hproj{\cC_1\opp\otimes\cC_2}))\ar@{<->}[r]^-{1:1}&
	\Hqe(\cC_1,\hproj{\cC_2})
	}
\end{equation}
which, roughly speaking sends the isomorphism class of an h-projective bimodule $E$ to the corresponding quasi-functor $\dfun{E}$. Observe that, in principle, the tensor product in \eqref{eqn:intHoms3} should be derived. But this is just a minor technical problem which, without loss of generality, we can ignore in this presentation.

It is not difficult to see that the image of the injection \eqref{eqn:intHoms2} consists of
all the morphisms $f\in\Hqe(\cC_1,\hproj{\cC_2})$ with the property that
$\im(\Ho(f))\subseteq\Ho(\essim{\cC_2})$. Hence, \eqref{eqn:intHoms3} provides a bijection
between the image of \eqref{eqn:intHoms2} and the set of isomorphism classes of the objects
$E\in\Ho(\hproj{\cC_1\opp\otimes\cC_2})$ such that
$\Ho(\dfun{E})\colon\Ho(\cC_1)\to\Ho(\hproj{\cC_2})$ factors through $\Ho(\essim{\cC_2})$. Thus we get \eqref{eqn:firstpart}.

\begin{remark}\label{rmk:dgFM}
The choice of the notation $\dfun{E}$, for an h-projective dg module $E\in\hproj{\cC_1\opp\otimes\cC_2}$ clearly suggests that the bijections \eqref{eqn:firstpart} should be thought of as the correct way to define a dg version of the notion of Fourier--Mukai functor discussed above. Indeed, it shows that we can naturally associate a(n isomorphism class of a) bimodule to a morphism in $\Hqe$. This should also be compared to the standard Morita theory which shows that any colimit preserving functor between the categories of modules over associative algebras (over a field $\K$) is realized by a bimodule.

Hence, \autoref{thm:Toen} shows that the analogue of \eqref{eqn:fun} at the level of objects exists and, contrary to the triangulated case, the natural bijection has all possible nice properties.
\end{remark}


Let us now move to the geometric reinterpretation of \autoref{thm:Toen} and, for simplicity, let us assume for the rest of this section that $X_1$ and $X_2$ are smooth projective schemes over a field $\K$. Let $\fF\colon\Db(X_1)\to\Db(X_2)$ be an exact functor and let $(\cC_i,\fE_i)$ be an enhancement of $\Db(X_i)$, for $i=1,2$.

\begin{definition}\label{def:dglift}
A morphism $f\in\Hqe(\cC_1,\cC_2)$ is a \emph{dg lift} (or simply a \emph{lift}) of $\fF$ if there is an isomorphism of exact functors $\fF\iso\fE_2\comp\Ho(f)\comp\fE_1^{-1}$.
\end{definition}

\begin{remark}\label{rmk:dglift}
It is quite easy to see that, in view of the strong uniqueness of the enhancements of $\Db(X_1)$ and $\Db(X_2)$ (see \autoref{thm:LOstrong}), the existence and uniqueness of a lift of $\fF$ do not depend on the choice of the enhancements of the triangulated categories. 
\end{remark}

So, in the above situation, denote by $(\Perf(X_i),\fE_i)$ any dg enhancement of $\Db(X_i)=\Dp(X_i)$.
We then get the following result (originally proved under milder assumptions).

\begin{thm}[\cite{To}, Theorem 8.15]\label{thm:Toen2}
If $X_1$ and $X_2$ are as above, then there is an isomorphism
\begin{equation}\label{eqn:isogeo}
\IHom(\Perf(X_1),\Perf(X_2))\to\Perf(X_1\times X_2)
\end{equation}
in $\Hqe$.
\end{thm}

\begin{remark}\label{rmk:BZ}
The proof of this fact involves the deep results in \cite{BB} about generation of derived categories of smooth projective schemes. This theorem was variously generalized in \cite{BZFN} and \cite{BZNP}, covering the more general setting of derived stacks.
\end{remark}

\autoref{thm:Toen2} and \autoref{thm:Toen} clarify even better the relation to the functor \eqref{eqn:fun} discussed in \autoref{rmk:dgFM}. Indeed, the morphisms in $\Hqe$ between any two enhancements of $\Db(X_1)$ and $\Db(X_2)$ are in natural bijection with the isomorphism classes of objects in $\Db(X_1\times X_2)$.
More precisely, consider the image $f\in\Hqe(\Perf(X_1),\Perf(X_2))$
of the object $E\in\Db(X_1\times X_2)$ under the bijection \eqref{eqn:firstpart} and the bijection induced by \eqref{eqn:isogeo}. It is then natural to ask whether $f$ is a dg lift of $\FM{E}$.
A positive answer is claimed in \cite{To} without a precise proof. A complete argument has been recently provided in \cite{LS}, under more general assumptions than the ones in this paper (see also \cite{Sn}). Putting altogether, we get the following result (see also Section 7 of \cite{BLL} for another proof of the `if' part).

\begin{prop}\label{prop:FMlift}
Let $X_1$ and $X_2$ be smooth projective schemes over a field $\K$ and let $\fF\colon\Db(X_1)\to\Db(X_2)$ be an exact functor. Then $\fF$ is of Fourier--Mukai type if and only if $\fF$ has a dg lift.
\end{prop}

\begin{remark}\label{rmK:funlift}
The discussion in \autoref{subsect:funtria} shows that, even when a dg lift of an exact functor exists, it need not be unique. The counterexamples in \autoref{subsect:funcounter} show that the lift may not exist.
\end{remark}


\bigskip

{\small\noindent{\bf Acknowledgements.} It is our great pleasure to thank Alice Rizzardo, Michel Van den Bergh and Vadim Vologodsky for answering all our questions about dg lifts of Fourier--Mukai functors. We are very grateful to Pieter Belmans and Valery Lunts for comments on a preliminary version of this paper. The anonymous referee suggested several useful improvements to the presentation. We warmly thank the organizers of the conference ``VBAC2015: Fourier-Mukai, 34 years on'' (Warwick, 2015) for the very exciting atmosphere. P.S.\ has benefited of many interesting conversations with Amnon Neeman and Olaf Schn\"{u}rer during the Hausdorff School ``Derived Categories: Dimensions, Stability Conditions, and Enhancements'' (Bonn, 2016).}



\begin{thebibliography}{99}

\bibitem{A} J.F.\ Adams, \emph{Stable homotopy and generalised homology}, Univ.\ Chicago Press, Chicago (1974).

\bibitem{AJ} L.\ Alonso Tarr\'{\i}o, A.\ Jerem\'{\i}as L\'opez, M.J.\ Souto Salorio, \emph{Localization in categories of complexes and unbounded resolutions}, Canad.\ J.\ Math.\ {\bf 52} (2000), 225--247.

\bibitem{AM} F.\ Amodeo, R.\ Moschetti, \emph{Fourier-Mukai functors and perfect complexes on dual numbers}, J.\ Algebra {\bf 437} (2015), 133--160.

\bibitem{AT} M.\ Anel, B.\ To\"{en}, \emph{D\'{e}nombrabilit\'{e} des classes d'\'{e}quivalences d\'{e}riv\'{e}es de vari\'{e}t\'{e}s alg\'{e}briques}, J.\ Algebraic Geom.\ {\bf 18} (2009), 257--277.


\bibitem{BKI} D.J.\ Benson, S.B.\ Iyengar, H.\ Krause, \emph{Stratifying modular representations of finite groups}, Ann.\ of Math.\ {\bf 174} (2011), 1643--1684.

\bibitem{BZFN} D.\ Ben--Zvi, D.\ Nadler, A.\ Preygel, \emph{Integral transforms and Drinfeld centers in Derived Algebraic Geometry}, J.\ Amer.\ Math.\ Soc.\ {\bf 23} (2010), 909--966.

\bibitem{BZNP} D.\ Ben--Zvi, J.\ Francis, D.\ Nadler, \emph{Integral transforms for coherent sheaves}, to appear in: J.\ Eur.\ Math.\ Soc., arXiv:1312.7164.

\bibitem{Bl} J.\ Block, \emph{Duality and equivalence of module categories in noncommutative geometry}, in: A celebration of the mathematical legacy of Raoul Bott, 311--339, CRM Proc.\ Lecture Notes {\bf 50}, Amer.\ Math.\ Soc., Providence, RI, 2010. 


\bibitem{BLL} A.\ Bondal, M.\ Larsen, V.\ Lunts, \emph{Grothendieck ring of pretriangulated categories}, Int.\ Math.\ Res.\ Not.\ {\bf 29} (2004), 1461--1495.

\bibitem{BB} A.\ Bondal, M. Van den Bergh, \emph{Generators and representability of functors in commutative and noncommutative geometry}, Moscow Math.\ J.\ {\bf 3} (2003), 1--36.



\bibitem{Cal} A.\ C\u{a}ld\u{a}raru, \emph{Derived categories of twisted sheaves on Calabi--Yau manifolds}, Ph.-D.\ thesis Cornell (2000).


\bibitem{CS3} A.\ Canonaco, P.\ Stellari, \emph{Fourier--Mukai functors: a survey},  EMS Ser.\ Congr.\ Rep., Eur.\ Math.\ Soc.\ (2013), 27--60.

\bibitem{CS} A.\ Canonaco, P.\ Stellari, \emph{Fourier-Mukai functors in the supported case}, Compositio Math.\ {\bf 150} (2014), 1349--1383.

\bibitem{CS4} A.\ Canonaco, P.\ Stellari, \emph{Internal Homs via extensions of dg functors}, Adv.\ Math.\ {\bf 277} (2015), 100--123.

\bibitem{CS1} A.\ Canonaco, P.\ Stellari, \emph{Non-uniqueness of Fourier--Mukai kernels}, Math.\ Z.\ {\bf 272} (2012), 577--588.

\bibitem{CS5} A.\ Canonaco, P.\ Stellari, \emph{Twisted Fourier--Mukai functors}, Adv.\ Math.\ {\bf 212} (2007), 484--503.

\bibitem{CS6} A.\ Canonaco, P.\ Stellari, \emph{Uniqueness of dg enhancements for the derived category of a Grothendieck category}, J.\ Eur.\ Math.\ Soc.\ {\bf 20} (2018), 2607--2641.

\bibitem{CS7} A.\ Canonaco, P.\ Stellari, \emph{Uniqueness of dg enhancements for the derived category of a Grothendieck category}, arXiv:1507.05509v5.


\bibitem{Ch} S.U.\ Chase, \emph{Direct products of modules}, Trans.\ Amer.\ Math.\ Soc.\ {\bf 97} (1960), 457--473.

\bibitem{Dr} V.\ Drinfeld, \emph{DG quotients of DG categories}, J.\ Algebra {\bf 272} (2004), 643--691.

\bibitem{DS} D.\ Dugger, B.\ Shipley, \emph{A curious example of triangulated-equivalent model categories which are not Quillen equivalent},  Algebr.\ Geom.\ Topol.\ {\bf 9} (2009), 135--166. 

\bibitem{GU} P.\ Gabriel, F.\ Ulmer, \emph{Lokal pr\"asentierbare Kategorien}, Lecture Notes in Math.\ {\bf 221}, Springer, Berlin, (1971).

\bibitem{Ge} F.\ Genovese, \emph{The uniqueness problem of dg-lifts and Fourier--Mukai kernels}, J.\ London Math.\ Soc.\ {\bf 94} (2016), 617--638.


\bibitem{HNR} J.\ Hall, A.\ Neeman, D.\ Rydh, \emph{One positive and two negative results for derived categories of algebraic stacks}, arXiv:1405.1888.

\bibitem{HR} J.\ Hall, D.\ Rydh, \emph{Perfect complexes on algebraic stacks}, arXiv:1405.1887.

\bibitem{Ho} M.\ Hovey, \emph{Model Categories}, Mathematical Surveys and Monographs {\bf 63}, American Mathematical Society, Providence, RI (1999), xii+209.

\bibitem{H} D.\ Huybrechts, \emph{Fourier--Mukai transforms in algebraic geometry}, Oxford Mathematical Monographs, Oxford Science Publications (2006).


\bibitem{Ka1} Y.\ Kawamata, \emph{D-equivalence and K-equivalence}, J.\ Differential Geom.\ {\bf 61} (2002), 147--171.

\bibitem{Ka} Y.\ Kawamata, \emph{Equivalences of derived categories of sheaves on smooth stacks}, Am.\ J.\ Math.\ {\bf 126} (2004), 1057--1083.

\bibitem{KTel} B.\ Keller, \emph{A remark on the generalized smashing conjecture}, Manuscripta Math.\ {\bf 84} (1994), 193--198.

\bibitem{Ke1} B.\ Keller, \emph{Deriving DG categories}, Ann.\ Sci.\ \'{E}cole Norm.\ Sup.\ {\bf 27} (1995), 63--102.

\bibitem{K} B.\ Keller, \emph{On differential graded categories}, International Congress of Mathematicians Vol.\ II, Eur.\ Math.\ Soc., Z\"urich (2006), 151--190.

\bibitem{Ko} M. Kontsevich, \emph{Homological algebra of Mirror Symmetry}, in: Proceedings of the International Congress of Mathematicians (Zurich, 1994, ed. S.D. Chatterji), Birkhauser, Basel (1995), 120-139.


\bibitem{K3} H.\ Krause, \emph{Derived categories, resolutions, and Brown representability}, in: Interactions between homotopy theory and algebra, 101--139, Contemp.\ Math.\ {\bf 436}, Amer.\ Math.\ Soc., Providence, RI, 2007.

\bibitem{K0} H.\ Krause, \emph{Deriving Auslander's formula}, Documenta Math.\ {\bf 20} (2015), 669--688.

\bibitem{K1} H.\ Krause, \emph{On Neeman's well generated triangulated categories}, Documenta Math.\  {\bf 6} (2001), 119--125.

\bibitem{K2} H.\ Krause, \emph{Localization theory for triangulated categories}, in: Triangulated categories, London Math.\ Soc.\ Lecture Note Ser.\ {\bf 375}, Cambridge Univ.\ Press (2010), 161--235.

\bibitem{Ksm1} H.\ Krause, \emph{Cohomological quotients and smashing localizations}, Amer.\ J.\ Math.\ {\bf 127} (2005), 1191--1246.

\bibitem{Ksm2} H.\ Krause, \emph{Smashing subcategories and the telescope conjecture--an algebraic approach}, Invent.\ Math.\ {\bf 139} (2000), 99--133.

\bibitem{LMB} G.\ Laumon, L.\ Moret-Bailly, \emph{Champs alg\'{e}briques}, Ergebnisse der Mathematik und ihrer Grenzgebiete.\ 3.\ Folge.\ A Series of Modern Surveys in Mathematics {\bf 39}, Springer-Verlag (2000), xii+208.

\bibitem{L} J.\ Lesieutre, \emph{Derived-equivalent rational threefolds}, Int.\ Math.\ Res.\ Notices {\bf 2015} (2015), 6011--6020.


\bibitem{LO} V.\ Lunts, D.\ Orlov, \emph{Uniqueness of enhancements for triangulated categories}, J.\ Amer.\ Math.\ Soc.\ {\bf 23} (2010), 853--908.

\bibitem{LS} V.\ Lunts, O.M.\ Schn\"urer, \emph{New enhancements of derived categories of coherent sheaves and applications}, J.\ Algebra {\bf 446} (2016), 203--274.


\bibitem{MSS} F.\ Muro,  S.\ Schwede, N.\ Strickland, \emph{Triangulated categories without models}, Invent.\ Math.\ {\bf 170} (2007), 231--241.

\bibitem{N3} A.\ Neeman, \emph{On the derived category of sheaves on a manifold}, Documenta Math.\ {\bf 6} (2001), 483--488.

\bibitem{N1} A.\ Neeman, \emph{The connection between the K-theory localization theorem of Thomason, Trobaugh and Yao and the smashing subcategories of Bousfield and Ravenel}, Ann.\ Sci.\ \'{E}cole Norm.\ Sup.\ {\bf 25} (1992), 547--566.


\bibitem{NInj} A.\ Neeman, \emph{The homotopy category of injectives}, Algebra Number Theory {\bf 8} (2014), 429--456.

\bibitem{N2} A.\ Neeman, \emph{Triangulated categories}, Annals of Mathematics Studies {\bf 148}, Princeton University Press (2001), viii+449.

\bibitem{Or} D.\ Orlov, \emph{Equivalences of derived categories and K3 surfaces}, J.\ Math.\ Sci.\ {\bf 84} (1997), 1361--1381

\bibitem{P} M.\ Porta, \emph{The Popescu-Gabriel theorem for
triangulated categories}, Adv.\ Math.\ {\bf 225} (2010), 1669--1715.

\bibitem{Pu} D.\ Puppe, \emph{On the stable homotopy category}, in: Proceedings of the International Symposium on Topology and its Applications (Budva, 1972), Savez Dru\v{s}tava Mat.\ Fiz.\ i Astronom.\ (1973), 200--212.

\bibitem{Qu} D.\ Quillen, \emph{Higher algebraic K-theory, I}, in: Algebraic K-theory, Lecture Notes in
Math. {\bf 341}, Springer (1973), 85--147.

\bibitem{Ra} D.C.\ Ravenel, \emph{Localization with respect to certain periodic homology theories}, Amer.\ J.\ Math.\ {\bf 105} (1984), 351--414.

\bibitem{RvdB} A.\ Rizzardo, M.\ Van den Bergh, \emph{An example of a non-Fourier--Mukai functor between derived categories of coherent sheaves}, arXiv:1410.4039.

\bibitem{R} R.\ Rouquier, \emph{Dimensions of triangulated categories}, J.\ K-theory {\bf 1} (2008), 193--258.

\bibitem{Sc} M.\ Schlichting, \emph{A note on K-theory and triangulated categories}, Invent.\ Math.\ {\bf 150} (2002), 111--116.

\bibitem{Sn} O.M.\ Schn\"urer, \emph{Six operations on dg enhancements of derived categories of sheaves}, arXiv:1507.08697.


\bibitem{S2} S.\ Schwede, \emph{Algebraic versus topological triangulated categories}, in: Triangulated categories, London Mathematical Society Lecture Notes {\bf 375}, Cambridge University Press, Cambridge (2010), 389--407.

\bibitem{S} S.\ Schwede, \emph{The $n$-order of algebraic triangulated categories}, J.\ Topol.\ {\bf 6} (2013), 857--867.

\bibitem{S1} S.\ Schwede, \emph{The $p$-order of algebraic triangulated categories}, J.\ Topol.\ {\bf 6} (2013), 868--914.

\bibitem{Sw} S.\ Schwede, \emph{Topological triangulated categories}, arXiv:1201.0899.

\bibitem{Se} C.\ Serp\'e, \emph{Resolution of unbounded complexes in Grothendieck categories},  J.\ Pure Appl.\ Algebra {\bf 177} (2003), 103--112. 
  
\bibitem{SP} The Stacks Project Authors, \emph{The Stacks Project}, {\tt http://stacks.math.columbia.edu/}.

\bibitem{Tab} G.\ Tabuada, \emph{Une structure de cat\'{e}gorie de mod\`{e}les de Quillen sur la cat\'{e}gorie des dg-cat\'{e}gories}, C.\ R.\ Math.\ Acad.\ Sci.\ Paris {\bf 340} (2005), 15--19.

\bibitem{TLec} B.\ To\"en, \emph{Lectures on dg-categories}, in: Topics in algebraic and topological K-theory, Lecture Notes in Math., Springer (2011), 243--302.

\bibitem{To} B.\ To\"en, \emph{The homotopy theory of dg-categories and derived Morita theory}, Invent.\ Math.\ {\bf 167} (2007), 615--667.

\bibitem{Tot} B.\ Totaro, \emph{The resolution property for schemes and stacks}, J.\ Reine Angew.\ Math.\ {\bf 577} (2004), 1--22.


\bibitem{Vo} C. Voisin, \emph{A counterexample to the Hodge conjecture extended to
K\"ahler varieties}, Int. Math. Res. Not. {\bf 20} (2002), 1057--1075.

\bibitem{V} V.\ Vologodsky, \emph{Triangulated endofunctors of the derived category of coherent sheaves which do not admit dg liftings}, arXiv:1604.08662.

\bibitem{We} C.\ Weibel, \emph{An introduction to homological algebra}, Cambridge University Press (1994).

\end{thebibliography}
\end{document}